\documentclass[12pt]{msml2021}  

\usepackage{nicefrac}
\usepackage{upgreek,amsfonts,bm, bbm,enumerate,multirow,diagbox,makecell}
\usepackage[subnum]{cases}
\usepackage[shortlabels]{enumitem} 
\usepackage{cite} 

\usepackage{algorithm,algpseudocode} 
\usepackage{algorithmicx,pseudocode}
\usepackage{relsize}  
\usepackage{cases} 
\usepackage{lipsum} 
\usepackage{caption} 
\usepackage{dblfloatfix} 

\newtheorem{thm}{Theorem} 
\newtheorem{lem}[theorem]{Lemma}

\newtheorem{prop}{Proposition}
\newtheorem{assumeA}{Assumption A.\hspace*{-1.2mm}}

\newtheorem{rem}{Remark}

\newtheorem{manualtheoreminner}{Assumption A.\hspace*{-1.2mm}} 
\newenvironment{assumeAprime}[1]{%
	\renewcommand\themanualtheoreminner{#1}%
	\manualtheoreminner
}{\endmanualtheoreminner}

\newtheorem{manuallemmaprime}{Lemma} 
\newenvironment{lemPrime}[1]{%
	\renewcommand\themanuallemmaprime{#1}%
	\manuallemmaprime
}{\endmanuallemmaprime}

\newcommand{\remove}[1]{}
\newcommand{\given}[2]{\left. #1 \right| #2}
\newcommand{\ilparenthesis}[1]{( #1 )}
\newcommand{\ilbracket}[1]{[ #1 ]} 
\newcommand{\ilset}[1]{\{ #1 \}}
\def \set#1{\left \{#1 \right \}}
\newcommand{\bracket}[1]{\left[  #1  \right]}

\def \parenthesis#1{\left (#1 \right)}
\newcommand{\norm}[1]{\| #1 \|} 
 
\newcommand{\abs}[1]{\left| #1 \right|}

\newcommand{\indicator}{\mathbb{I}}

\newcommand{\diff}{\mathrm{d}} 

\newcommand{\real}{\mathbb{R}}

\newcommand{\field}{\mathcal{F}}

\usepackage{bbm}
\newcommand{\Prob}{\mathbbm{P}}
\newcommand{\E}{\mathbbm{E}}
\newcommand{\Var}{\mathrm{Var}}

\usepackage{relsize}
\newcommand{\transpose}{\mathsmaller{T}}
\newcommand{\tr}{\mathrm{tr}}
\newcommand{\sign}{\mathrm{sign}}

\usepackage{bm}
 
\newcommand{\zero}{\boldsymbol{0}}

\newcommand{\bA}{\bm{A}}
\newcommand{\bB}{\bm{B}}
\newcommand{\bC}{\bm{C}}
\newcommand{\bD}{\bm{D}}

\newcommand{\bH}{\bm{H}}
\newcommand{\bI}{\bm{I}}

\newcommand{\bU}{\bm{U}}

\newcommand{\bW}{\bm{W}}
\newcommand{\bX}{\bm{X}}

\newcommand{\bd}{\bm{d}}

\newcommand{\bg}{\bm{g}}

\newcommand{\bu}{\bm{u}}
\newcommand{\bv}{\bm{v}}

\newcommand{\bx}{\bm{x}}

\usepackage{upgreek}
\renewcommand{\alpha}{\upalpha}
\renewcommand{\beta}{\upbeta}
\renewcommand{\gamma}{\upgamma}
\renewcommand{\delta}{\updelta}
\renewcommand{\epsilon}{\upepsilon}
\renewcommand{\varepsilon}{\upvarepsilon}
\renewcommand{\zeta}{\upzeta}
\renewcommand{\eta}{\upeta} 
\renewcommand{\theta}{\uptheta}
\renewcommand{\vartheta}{\upvartheta}
\renewcommand{\iota}{\upiota}
\renewcommand{\kappa}{\upkappa}
\renewcommand{\lambda}{\uplambda}
\renewcommand{\mu}{\upmu}
\renewcommand{\nu}{\upnu}
\renewcommand{\xi}{\upxi}
\renewcommand{\pi}{\uppi}
\renewcommand{\rho}{\uprho}
\renewcommand{\varrho}{\upvarrho}
\renewcommand{\sigma}{\upsigma}
\renewcommand{\tau}{\uptau}
\renewcommand{\upsilon}{\upupsilon}
\renewcommand{\phi}{\upphi}
\renewcommand{\varphi}{\upvarphi}
\renewcommand{\chi}{\upchi}
\renewcommand{\psi}{\uppsi}
\renewcommand{\omega}{\upomega} 

\newcommand{\bbeta}{\boldsymbol{\upbeta}}

\newcommand{\bDelta}{\boldsymbol{\Delta}} 

\newcommand{\bzeta}{\boldsymbol{\upzeta}}
\newcommand{\btheta}{\boldsymbol{\uptheta}}
\newcommand{\bvartheta}{\boldsymbol{\uptheta}^* }

\newcommand{\bLambda}{\boldsymbol{\Lambda}}

\newcommand{\bGamma}{\boldsymbol{\Gamma}}
\newcommand{\bmu}{\boldsymbol{\upmu}} 

\newcommand{\bxi}{\boldsymbol{\upxi}}

\newcommand{\bSigma}{\boldsymbol{\Sigma}}

\newcommand{\bPsi}{\boldsymbol{\Psi}}

\usepackage{mathrsfs}
\newcommand{\gain}{a}
\newcommand{\perturb}{c} 
\newcommand{\direction}{\bu }
\newcommand{\tperturb}{\widetilde{c}}
\newcommand{\loss}{f}
\newcommand{\lossNoisy}{F}   
\newcommand{\hbg}{\hat{\bg}}
\newcommand{\hbtheta}{\hat{\btheta}}
 
\newcommand{\bias}{\bbeta}  
\newcommand{\noise}{\bxi}  
\newcommand{\LipsPara}{\mathrm{L}}
\newcommand{\convexPara}{\mathrm{C}}
\newcommand{\algoName}[1]{\textsc{#1}} 

\newcommand{\obH}{\overline{{\bH}}}
\newcommand{\hbH}{\hat{\bH}}
 
\newcommand{\mappingPD}{\bm{p}} 
\newcommand{\Dimension}{d}

\newcommand{\funDensity}{p}
\newcommand{\funTest}{q}
\newcommand{\radius}{\updelta}
\newcommand{\correction}{\upvarepsilon}

\newcommand{\sampleIndex}{i}
\newcommand{\sampleTotal}{I}
\newcommand{\sampleBatch}{J}

\newcommand{\BoundHessianMoment}{\updelta}
\newcommand{\normalDist}{\mathrm{N}}

\usepackage{xcolor} 

\usepackage{mathtools}
\newcommand{\gainH}{w}
\newcommand{\gainHweighted}{\overline{\gainH}}
\newcommand{\lossNoise}{\upvarepsilon}
\newcommand{\obg}{\overline{\bg}}
\newcommand{\iterTotal}{K}
\newcommand{\obtheta}{\overline{\btheta}}

\title[Hessian Estimation via Stein's Identity  in Black-Box Problems]{Hessian Estimation via Stein's Identity  in Black-Box Problems} 
\usepackage{times} 
\msmlauthor{\Name{Jingyi Zhu} \Email{jingyi.zhu@alibaba-inc.com}\\ 
	\addr DAMO Academy, Alibaba Inc.,  Bellevue, WA }

\makeatletter
\let\Ginclude@graphics\@org@Ginclude@graphics
\makeatother

\begin{document}
	
	\maketitle
	
	\begin{abstract}

 When the available information is noisy zeroth-order (ZO) oracle,  stochastic approximation   methods are popular for estimating the root of the multivariate gradient equation. Inspired by  the  Stein's identity, this work  establishes  a novel Hessian approximation scheme. 
 				We compare it alongside  second-order simultaneous perturbation stochastic approximation  (\algoName{2SPSA}) in  \citep{spall2000adaptive}.  On the basis of the almost sure convergence and the  \emph{same} convergence rate, 
 				\algoName{2SPSA} requires \emph{four} ZO queries, while ours requires \emph{three} ZO queries. Moreover, \algoName{2SPSA} requires \emph{two} statistically independent perturbations and \emph{two} differencing stepsizes, while ours  requires generating \emph{one} perturbation vector only and tuning \emph{one} differencing stepsize only. 
 				 					Besides, the weighting mechanism for the Hessian estimate is generalized and the smoothness restriction on the loss function is relaxed compared to \algoName{2SPSA}.  
	Finally, we present numerical support for the reduced per-iteration ZO query complexity.

	\end{abstract}
	
	\begin{keywords}  stochastic
	optimization; Hessian estimation; Stein's identity; gradient-free methods
	\end{keywords}

	\section{Introduction}  
 Stochastic approximation (SA) has been widely applied in the stochastic  black-box (a.k.a. derivative-free) problem. 
	Let $\btheta\in\real^{\Dimension}$ concatenate all the adjustable model parameters.  
	Let the system stochasticity be represented by the random variable $\upomega$ following a probability distribution $\Prob$ on its domain $\Omega$.     Consider finding the minimizer for  a twice-differentiable bounded-from-below loss function $ \loss (\cdot ): \real^p \to \real $:
	\begin{equation}\label{eq:SAsetup}
 \bvartheta \equiv 	\arg\min_{\btheta\in\real^{\Dimension}} \loss\ilparenthesis{\btheta}\,, \text{ where } \loss\ilparenthesis{\btheta} \equiv 
		\E_{\upomega\sim\Prob  } \bracket{ \lossNoisy \ilparenthesis{\btheta, \upomega}}\,. 
	\end{equation}
	In (\ref{eq:SAsetup}), $ \lossNoisy\ilparenthesis{\cdot,\cdot}:\real^{\Dimension}\times\Omega\mapsto\real $ evaluated at $ \ilparenthesis{\btheta,\upomega} $ represents one \emph{noisy} observation   of $\loss\ilparenthesis{\btheta}$ corrupted by  $\upomega$.   It is common that these zeroth-order (ZO) queries $ \lossNoisy\ilparenthesis{\cdot,\cdot} $ are \emph{expensive} to evaluate. Under the aforementioned setup,  the generic SA recursions are widely  applied:  
	\begin{subnumcases}{\label{eq:SA} \hbtheta_{k+1} = 	 } \hbtheta_k - \gain_k \hbg_k  \,, \hspace{.3in} \quad  \text{stochastic 1st-order  method using ZO oracles,}  \label{eq:SA1st} 
\\ \hbtheta_k - \gain_k \hbH_k^{-1} \hbg_k \,, \quad \text{stochastic 2nd-order  method using ZO oracles,}\quad  \label{eq:SA2nd}
\end{subnumcases} 
where $\hbtheta_k$ denotes the recursive estimate at the $k$th iteration, $\gain_k$ is positive stepsize (gain),  $ \hbg_k $ represents an estimate for the gradient function $\bg\ilparenthesis{\btheta}\equiv\nabla\loss\ilparenthesis{\btheta}$ evaluated at $\hbtheta_k$, and $\hbH_k $ represents   an estimate for the Hessian function $\bH\ilparenthesis{\btheta} \equiv \nabla^2\loss\ilparenthesis{\btheta}$ evaluated at $\hbtheta_k$. Besides SA, random search (including stochastic ruler, stochastic comparison, simulated
annealing, etc.) is also useful in solving (\ref{eq:SA}), but it will not be our focus.

	\subsection{Why Do We Care Second-Order SA in Black-Box Problems?}\label{sect:damping}
The SA  algorithms (\ref{eq:SA}) \emph{in part}\footnote{   Subtleties relating to both  the noisy ZO oracles and the stepsize are  fundamental to enable (\ref{eq:SA}) to be effective for (\ref{eq:SAsetup}).  } stem  from the localized model $	 \bracket{\loss\ilparenthesis{\btheta} + \bd^\transpose\bg\ilparenthesis{\btheta }  +  \bd ^\transpose\bB\ilparenthesis{\btheta} \bd /2 } $ constructed within a  neighborhood $ \ilset{\btheta+\bd: \norm{\bd}\le \radius} $, where $ \bB\ilparenthesis{\cdot} $ is some curvature matrix. In a nutshell, letting $\bB\ilparenthesis{\btheta}=\LipsPara_2  \bI$ (where $\LipsPara_2$ upper-bounds the spectral norm of the Hessian function $\bH\ilparenthesis{\cdot}$ over the domain) and  $ \bB\ilparenthesis{\btheta} = \bH\ilparenthesis{\btheta} $ motivate   (\ref{eq:SA1st}) and (\ref{eq:SA2nd}), respectively.

Currently,   (\ref{eq:SA1st}) remains dominant over  (\ref{eq:SA2nd}), largely because (\ref{eq:SA2nd}) suffers from       (i) the model-trust region/radius issue and (ii) expensive computational cost. The issue (i) becomes severe when we naively implement (\ref{eq:SA2nd}) around a region with negative curvature for \emph{nonconvex} loss functions. On the contrary, issue (i) does not arise in   (\ref{eq:SA1st}) because  the conservative local model ($ \bB\ilparenthesis{\btheta} = \LipsPara_2\bI $) enforces \emph{extremely  small} updates. 
Besides, large parameter space $\Dimension$  and/or large  dataset size $\sampleTotal$ in issue (ii)  may       deter (\ref{eq:SA2nd})  to be applied in practical     problems.   
  The storage, update, and inversion of the $\Dimension^2$-entry curvature matrix $\bB\ilparenthesis{\hbtheta_k}$ may  be prohibitive\footnote{	Large $\sampleTotal$ makes it costly to evaluate Hessian, and large $\Dimension$ makes it costly to invert the Hessian.  Each Newton iteration requires $O(\sampleTotal\Dimension^2)$ to evaluate the \emph{exact} Hessian (for the sample) and $O(\Dimension^3)$ to invert it.  }.   
  	Fortunately, using the damping  techniques in Levenberg-Marquardt method  {(see Remark~\ref{rem:damping})},   issue  (i) can be largely overcome.  To address     issue (ii), researchers may impose low-rank, diagonal, and other structures onto the curvature matrix   for easily storing/computing/inverting $\bB(\cdot)$.  
  	\begin{rem}
  		\label{rem:damping} 
  		Computing  $
  		\arg\min_{\bd: \norm{\bd}\le \radius} \bracket{ \loss\ilparenthesis{\btheta} + \bd^\transpose \bg\ilparenthesis{\btheta}  + \bd^\transpose\bB \ilparenthesis{\btheta} \bd  /2} $ is equivalent to computing $
  		\arg\min_{\bd \in\real^{\Dimension} } \bracket{ \loss\ilparenthesis{\btheta} + \bd^\transpose  \bg\ilparenthesis{\btheta}  +  \bd^\transpose\parenthesis{\bB \ilparenthesis{\btheta}+\correction\bI } \bd/2 } \equiv - \ilparenthesis{\bB\ilparenthesis{\btheta} + \correction\bI}^{-1} \bg\ilparenthesis{\btheta} $. Although 	$\correction$ is a complicated function of $\radius$,    we can   work with $\correction$ \emph{directly}.
  	\end{rem}
  
	The   estimates from   (\ref{eq:SA1st})   suffer from slow convergence rate in later iterations after  a sharp decline during early iterations. When $\hbtheta_k$ is in the vicinity of $\bvartheta$,  the algorithmic scheme (\ref{eq:SA2nd}) offers multiple edges:  (a)  resulting $\hbtheta_k$   remains intact under \emph{linear} mappings (imposed on $\btheta$);  (b) they avoid\footnote{For   (\ref{eq:SA1st}) using ZO queries,  the gain sequence in the form of $ \gain_k = \gain/k $ must satisfy $ \gain >\ilbracket{  3\uplambda_{\min} \ilparenthesis{\bH\ilparenthesis{\bvartheta}}}^{-1} $ \citep{spall1992multivariate} in order to attain the fastest convergence rate of $ O(k^{-\nicefrac{1}{3}}) $.  Second-order methods (\ref{eq:SA2nd}) can achieve the  \emph{optimum} rate \emph{without}  
	 	knowing the minimum eigenvalue of $\bH\ilparenthesis{\bvartheta}$. } the need for hyper-parameter tuning; (c) faster convergence when the iterate $\hbtheta_k$ is close to the optima ({defined as $\bvartheta$ in (\ref{eq:SA})}) where the underlying loss function  $\loss\ilparenthesis{\cdot}$ is locally quadratic; (d) local curvature exploitation (preconditioning) makes      the loss surface more isotropic\footnote{In stochastic optimization, the asymptotic behaviors of the stochastic optimizer only cares about how many data you've seen. The optimization problem becomes an estimation problem. 
	 	The asymptotic convergence behavior of stochastic methods, which cares only about amount of data seen, kicks in sooner when second-order information are approximated, since navigating the \emph{curved} loss landscape stops being the bottleneck (but in well-conditioned problems it's already not the bottleneck.) } and mitigate the     ill-conditioning effects.

	\subsection{Perspectives on Randomness  and Data  } \label{sect:data}
	{To facilitate  discussion in Sect.~\ref{sect:prior}, Sect.~\ref{sect:algo}, and Sect.~\ref{sect:classification},  we comment on  the randomness  $\Omega\ni \upomega\sim\Prob $ in (\ref{eq:SAsetup}) that  inevitably arise when the function measurements are collected from either physical experiments or computer simulation.}  
	
	\paragraph{Data-Stream (Fresh Samples)}
	For online learning where fresh  sample arrives sequentially, $\upomega\in\Omega$ may represent some additive/multiplicative noise imposed on the underlying loss.
	When data come in a streaming fashion,  it is natural  to assume that all the  sample points  $\upomega$'s are independently drawn from $\Omega$ and identically distributed according to the distribution  $\Prob$.

	\paragraph{Fixed Dataset (Collected Samples)} For empirical risk minimization (ERM) where only a fixed number (say $\sampleTotal$) of (training) samples\footnote{Although the experimenter only has  access to the  \emph{training} error (based on the fixed  $\sampleTotal$ samples)   in black-box problems,  the ultimate goal for ERM should be minimizing the \emph{generalization} error. Note that   \emph{deterministic} optimization applied on given dataset (deeming the empirical risk loss  as a deterministic loss function) does \emph{not}  work in terms of the \emph{generalization} performance. } are available, it is standard   practice to \emph{reuse}  the given $\sampleTotal$ samples using    a  (mini-)batch size of $\sampleBatch$ ($1\le \sampleBatch\le \sampleTotal$).  When sampling \emph{with} replacement is applied, all the $\sampleBatch$-elements subsets of $ \set{1,\cdots,\sampleTotal} $ constitute $\Omega$, and
	$\Prob$ places   a probability mass of $ \nicefrac{\sampleBatch!(\sampleTotal-\sampleBatch)!}{\sampleTotal!} $ on each $\omega\in\Omega$. Unlike sampling \emph{without} replacement (exhausting every sample within one epoch), this  scenario where $ \omega\in\Omega $ following uniform distribution $\Prob$ can still  fit into   (\ref{eq:SAsetup}). 
	
	\remove{    When sampling without replacement is applied,  the experimenter will use every sample within one epoch. }

  \subsection{Prior Work on Second-Order SA Using ZO Queries}\label{sect:prior}

  Second-order estimation using ZO oracles can be  traced back to \citep{fabian1971stochastic}, which requires $ O(\Dimension^2) $ ZO queries   per iteration. \citep{spall2000adaptive} proposed a simultaneous perturbation version of the Hessian estimate, which costs \emph{four} ZO queries  per iteration. Later \citep{bhatnagar2015simultaneous} considered a similar estimation form using \emph{three} ZO queries per iteration, yet the Hessian estimator involves several   contrived constants.      {We note   that    \citet{martens2015optimizing,wang2017stochastic,agarwal2019efficient}  use first-order oracle,   \citet{agarwal2017second}  uses Hessian--vector-product oracle, and \citet{sohl2014fast,byrd2016stochastic,saab2019multidimensional} use  second-order oracle.} Here we  assume that only noisy ZO oracle  is  available, so   methods using other oracles  are beyond the scope of our comparison/discussion.

  \paragraph{Core Budget Indicator}
  For problem (\ref{eq:SAsetup}), the     \emph{ZO query complexity} (to achieve certain level of accuracy) is usually the \emph{key}  budget indicator.  When the number of ZO queries per iteration is fixed,  the \emph{iteration/runtime complexity} may  be used interchangeably with oracle complexity in performance measure.  
  Besides, the floating-point-operations (FLOPs) per iteration may also be important for experimenters in high-dimensional problems, see \citet{zhu2019efficient}.

 \paragraph{Convergence Notion} When the loss function in ERM is composed of a finite   number of summands, notions of convergence and rates of convergence are in line with those in deterministic optimization.  For example, in \citet{johnson2013accelerating, martens2015optimizing, byrd2016stochastic, sohl2014fast, schraudolph2007stochastic}, rates of convergence are linear or quadratic as a measure of \emph{iteration-to-iteration improvement}  in the empirical risk function. In contrast, we follow the traditional SA notion, including applicability to general noisy loss functions (discussed in Sect.~\ref{sect:data}),   and stochastic notions of convergence and rates of convergence based on \emph{sample-points (almost surely, a.s.)} and convergence in distribution.

  \subsection{Overview and   Contribution}
 \label{sect:contribution}
  
  The remainder of this paper is organized as follows. Sect.~\ref{sect:Motivation}  conveys the motivation behind our novel Hessian approximation based on Stein's identity and presents implementation details. Sect.~\ref{sect:Theory}
   provides theoretical justification for the proposed second-order (damped) methods.  Sect.~\ref{sect:Numerical} illustrates the numerical performance and \algoName{2SPSA}.  Sect.~\ref{sect:Concluding} lists some concluding remarks and envisions some future directions. Before proceeding, let us outline the key  difference between ours and \algoName{2SPSA}.
  
 \begin{enumerate}[1)]
 	\item  On the basis that  the fastest convergence rate remains at $ O(k^{-\nicefrac{1}{3}}) $,   the Hessian estimate (\ref{eq:Hhat-3}) to appear only requires \emph{three} ZO queries, whereas \algoName{2SPSA} requires \emph{four} ZO queries.  Moreover, (\ref{eq:Hhat-3}) only requires generating \emph{one} perturbation vector via Monte Carlo and tuning one differencing stepsize, while \algoName{2SPSA} requires \emph{two} statistically independent perturbations and \emph{two} differencing stepsizes.

 	\item    The smoothing rate $\gainH_k$ in (\ref{eq:Newton-2}) is allowed to decrease as long as A.\ref{assume:Stepsize} is met, whereas \algoName{2SPSA} considers direct averaging only.

 	\item 
 	Our Hessian estimator is  symmetric by construction, while \algoName{2SPSA} in \citet{spall2000adaptive} requires an extra to manually symmetrize its Hessian estimate.   Moreover, our estimator has a rather clean and elegant form compared to \algoName{2RDSA} in \citet{prashanth2016adaptive}.

\item  The restriction of  ``four-times continuously differentiable with bounded fourth-order derivatives'' in \algoName{2SPSA} is changed to  ``thrice continuously differentiable with Lipschitz-continuous third-order derivatives.'' 

\item 
Based on Stein's identity, the motivation behind our gradient/Hessian estimators and the corresponding derivation for bias/variance are greatly simplified compared to \algoName{2SPSA}.

 \end{enumerate}

	\section{Motivation,  Description, and Implementation}\label{sect:Motivation}

	\subsection{Stein's Identity Motivates Our  Gradient/Hessian Estimators}
	
 {Note that the core to (\ref{eq:SA})  are the  gradient and Hessian estimators. Both  \citet{spall2000adaptive}  and \citet{prashanth2016adaptive}
	resort to the fundamental Taylor's theorem,  and they require ``four-times continuously differentiable with bounded fourth-order derivatives''  to achieve the optimal convergence rate of $ O(k^{-\nicefrac{1}{3}}) $. On the contrary,  Stein's identity allows us to construct gradient/Hessian estimators for a general class of   smooth functions, and enables the optimal convergence rate when the underlying function is ``thrice continuously differentiable with Lipschitz-continuous third-order derivatives.''
	The following proposition describes the basic Stein's identity.   We mention that   \citet{erdogdu2015newton} used Stein's identity to estimate Hessian too,   their work is restricted for \emph{linear} predictor with Gaussian data. }

	\begin{prop}
		[First-Order and Second-Order Stein's Identity] \label{prop:Stein}  \citep{stein2004use} Let $\bX \in\real^{\Dimension}$ be a $\Dimension$-dimensional  random vector whose underlying  density function  is $\funDensity\ilparenthesis{\cdot}:\real^{\Dimension}\mapsto\real$.

		\begin{enumerate}[i)]
			\item Assume that the density function $ \funDensity\ilparenthesis{\bx} $ is   \emph{differentiable}.  Let $\funTest:\real^{\Dimension}\mapsto\real$ be a differentiable  function such that $ \E\ilbracket{\nabla \funTest\ilparenthesis{\bX}} $ \emph{exists}. Then 
			\begin{equation}\label{eq:Stein1}
			\E\set{  \funTest(\bX ) \ilbracket{\funDensity\ilparenthesis{\bX}}^{-1 }    \nabla \funDensity\ilparenthesis{\bX } } = -\E\ilbracket{\nabla \funTest\ilparenthesis{\bX }}\,. 
			\end{equation}
			\item Assume that the density function $\funDensity\ilparenthesis{\bx}$ is \emph{twice differentiable}. Let $\funTest:\real^{\Dimension}\mapsto\real$ be a twice differentiable function such that $ \E\ilbracket{\nabla^2 \funTest\ilparenthesis{\bX}} $ \emph{exists}. Then \begin{equation}\label{eq:Stein2}
			\E\set{  \funTest(\bX ) \ilbracket{\funDensity\ilparenthesis{\bX}}^{-1}    \nabla^2 \funDensity \ilparenthesis{\bX } } =  \E\ilbracket{\nabla^2 \funTest\ilparenthesis{\bX }}\,. 
			\end{equation}
		\end{enumerate}
		
	\end{prop}
	 For  the special case of multivariate standard normal vector $\bX \sim\normalDist \ilparenthesis{\zero,\bI}$, we have   $ \nabla \funDensity\ilparenthesis{\bx } = - \bx    \funDensity\ilparenthesis{\bx  } $ and $ \nabla^2 \funDensity\ilparenthesis{\bx} =  \ilparenthesis{\bx\bx^\transpose-\bI }   \funDensity(\bx)$. In this case,    (\ref{eq:Stein1}) and (\ref{eq:Stein2})  reduce to 
	 \begin{equation*}
	 \begin{cases}
	 \E\ilbracket{\bX    \funTest\ilparenthesis{\bX}} = \E\ilbracket{\nabla \funTest\ilparenthesis{\bX }} \,, \\
	  \E \ilbracket{\ilparenthesis{\bX\bX^\transpose-\bI}    \funTest\ilparenthesis{\bX}} = \E\ilbracket{\nabla^2 \funTest\ilparenthesis{\bX }} \,, 
	 \end{cases}  \text{ when }\bX \sim\normalDist\ilparenthesis{\zero,\bI}\,. 
	 \end{equation*}
	  For the  case of $\bX\sim\normalDist\ilparenthesis{\zero,\bSigma}$, we have $ \nabla \funDensity\ilparenthesis{\bx} = - \bSigma^{-1}  \bx   \funDensity\ilparenthesis{\bx} $, $ \nabla^2 \funDensity\ilparenthesis{\bx} = \parenthesis{\bSigma^{-1}\bx\bx^\transpose\bSigma^{-1} - \bSigma^{-1} }    \funDensity(\bx) $. In this case,    (\ref{eq:Stein1}) and (\ref{eq:Stein2})  reduce to 
	   \begin{equation*}
	  \begin{cases}
\E\ilbracket{\bSigma^{-1}\bX   \funTest\ilparenthesis{\bX}} = \E\ilbracket{\nabla \funTest\ilparenthesis{\bX }}  \,, \\
	 \E\bracket{ \parenthesis{\bSigma^{-1} \bX\bX^\transpose\bSigma^{-1}  - \bSigma^{-1}}   \funTest (\bX ) } = \E \bracket{\nabla^2 \funTest\ilparenthesis{\bX }}  \,, 
	  \end{cases}  \text{ when }\bX \sim\normalDist\ilparenthesis{\zero,\bSigma}\,. 
	  \end{equation*}

\begin{rem}\label{rem:distExtensionCont}
			The distribution for the \emph{absolutely continuous} $\Dimension$-dimensional random vector $\bX$  is allowed to come  from  the  exponential family (with continuous density function) per \citet{hudson1978natural}, and the   family of elliptical distributions   (i.e.,  spherical distribution,  hyperbolic distribution, logistic distribution, multivariate Laplace distribution, multivariate $ t $-distribution) per \citet{landsman2008stein}.
	
\end{rem}

\begin{rem}\label{rem:distExtensionDisc}
	\citet{teerapabolarn2013stein} extends (\ref{eq:Stein1}--\ref{eq:Stein2}) to  several \emph{discrete} distributions, including the   binomial (Rademacher)  and  the Poisson distributions.
\end{rem}

	 {Let us introduce the   smoothed loss function $ \loss_{\perturb} \ilparenthesis{\btheta}$ constructed through the convolution  of the underlying loss function $\loss\ilparenthesis{\cdot}$ and the Gaussian kernel. 
		The upcoming gradient/Hessian estimator with perturbation magnitude $\perturb$ will be the \emph{unbiased} estimator for the gradient/Hessian of the \emph{smoothed} loss function $ \loss_{\perturb}\ilparenthesis{\cdot} $. 
		Even though $\loss\ilparenthesis{\cdot}$   may \emph{not} be continuous, $\loss_{\perturb}\ilparenthesis{\cdot}$ is  \emph{infinitely many times differentiable},  inheriting from the  Gaussian kernel function.  }

	 \begin{lem}[Stein's Identity Based Estimator]
	 	\label{lem:Stein} 
	Let $\direction\sim\normalDist\ilparenthesis{\zero,\bI}$. Let   $ \loss_{\perturb}\ilparenthesis{\btheta}\equiv \E _{\direction} \bracket{\loss\ilparenthesis{\btheta + \perturb\direction}} $ be  the \emph{smoothed} loss function. Assume the left-hand-sides of (\ref{eq:SteinGradient}--\ref{eq:SteinHessian3}) \emph{exist}, we have the following estimators. 
	\begin{enumerate}[i)]
		\item One-measurement estimator for $\nabla\loss _\perturb\ilparenthesis{\btheta}$ and $\nabla^2\loss_{\perturb}\ilparenthesis{\btheta}$:	\begin{equation}
		\label{eq:SteinGradient}
		\E _{\direction} \bracket{  \perturb^{-1} \direction {\loss\ilparenthesis{\btheta+\perturb\direction}}              } = \E_{\direction} \bracket{\bg  \ilparenthesis{\btheta+\perturb\direction}} \equiv \nabla\loss_{\perturb}\ilparenthesis{\btheta}    \,,
		\end{equation}
		\begin{equation}
		\label{eq:SteinHessian}  \E _{\direction} \bracket{ \perturb^{-2}\parenthesis{\direction\direction^\transpose-\bI }{\loss\ilparenthesis{\btheta+ \perturb\direction}}       }= 	\E_{\direction}\bracket{\bH\ilparenthesis{\btheta+\perturb\direction}} \equiv  \nabla^2 \loss_{\perturb}\ilparenthesis{\btheta}\,.
		\end{equation}

		\item Two-measurements estimator for $\nabla \loss_\perturb\ilparenthesis{\btheta}$ and $\nabla^2\loss_{\perturb}\ilparenthesis{\btheta}$: \begin{subnumcases}{\label{eq:SteinGradient2}  	\nabla\loss_{\perturb}\ilparenthesis{\btheta}  =}
		\E _{\direction} \set{\perturb^{-1}\ilbracket{\loss\ilparenthesis{\btheta+\perturb\direction}  - \loss\ilparenthesis{\btheta}} \direction   }\,,   \label{eq:SteinGradient2-1} \\
		\E _{\direction} \set{  (2\perturb)^{-1}\ilbracket{ \loss\ilparenthesis{\btheta+\perturb\direction} - \loss\ilparenthesis{\btheta-\perturb\direction} } \direction   } \,,  \label{eq:SteinGradient2-2}
		\end{subnumcases}
		\begin{subnumcases}{\label{eq:SteinHessian2}  		\nabla^2\loss_{\perturb}\ilparenthesis{\btheta} =}
		\E _{\direction} \set{ \perturb^{-2} \ilbracket{\loss\ilparenthesis{\btheta+\perturb\direction} - \loss\ilparenthesis{\btheta}}  \ilparenthesis{\direction\direction^\transpose-\bI }   } \,,   \label{eq:SteinHessian2-1} \\
		\E_{\direction} \set{ (2\perturb^2)^{-1} \ilbracket{\loss\ilparenthesis{\btheta+\perturb\direction} +  \loss\ilparenthesis{\btheta-\perturb\direction}} \ilparenthesis{\direction\direction^\transpose-\bI }   }\,.  \label{eq:SteinHessian2-2}
		\end{subnumcases}
		
		\item Three-measurements estimator for    $\nabla^2\loss_{\perturb}\ilparenthesis{\btheta}$:	\begin{equation}\label{eq:SteinHessian3}
		\nabla^2\loss_{\perturb}\ilparenthesis{\btheta} = \E_{\direction} \set{ (2\perturb^2)^{-1} \ilbracket{\loss\ilparenthesis{\btheta+\perturb\direction} +  \loss\ilparenthesis{\btheta-\perturb\direction} - 2\loss\ilparenthesis{\btheta}}\ilparenthesis{\direction\direction^\transpose-\bI }   }\,. 
		\end{equation}
	\end{enumerate}

\end{lem}

 	Although Lemma~\ref{lem:Stein}  does  \emph{not} provide an unbiased gradient/Hessian estimators for the true  loss function $\loss\ilparenthesis{\cdot}$, they are instrumental to construct an  \emph{asymptotically unbiased} estimators in (\ref{eq:ghat}--\ref{eq:Hhat}) with its bias vanishing at certain rate to be discussed in Lemmas~\ref{lem:BiasGradient}--\ref{lem:BiasHessian}.

	\subsection{Algorithmic Form}\label{sect:algo}
Recall that the model-trust issue discussed in 	Sect.~\ref{sect:damping}   can be handled by  the \emph{damping} technique   per Remark~\ref{rem:damping}. Therefore, we modify (\ref{eq:SA2nd}) as follows:    	
		\begin{subnumcases}{\label{eq:Newton}  	 }
 \hbtheta_{k+1} = \hbtheta_k-\gain_k \ilbracket{\mappingPD_k\ilparenthesis{\obH_k}}^{-1} \hbg_k \,, \label{eq:Newton-1} 
\\
 	\obH_{k+1} =  \obH_k - \gainH_k \ilparenthesis{ \obH_k - \hbH_{k}  }   \,,
  \label{eq:Newton-2}
	\end{subnumcases} for $k\ge 0$.  
The general adaptive SA algorithm 	(\ref{eq:Newton}) is comprised of two recursions: (\ref{eq:Newton-1}) estimates  $\btheta$ via   a stochastic analogue of the Newton method, and  (\ref{eq:Newton-2}) produces estimate for  $\bH \ilparenthesis{\cdot}$ through  a weighted average of the seen Hessian estimates $\hbH_k$'s. 
The stepsize $\gainH_k$ governing the smoothing rate for Hessian estimate places a crucial role in the convergence of (\ref{eq:Newton}),   and one  common choice\footnote{This is the stepsize enforced in \algoName{2SPSA}. Here we allow  $\gainH_k$ to go to zero at a rate such that $ \sum_k \gainH_k^2\perturb_k^{-4} <\infty $, see A.\ref{assume:Stepsize} to appear. } is $ \gainH_k  = \nicefrac{1}{\ilparenthesis{k+1}} $ for $k\ge 0$. When $\gainH_0>0$, the initialization $\obH_0$ can be  a scaling  matrix (identity multiplied by a scalar), so that early iteration of (\ref{eq:Newton}) resembles that of (\ref{eq:SA1st}).  
	 In (\ref{eq:Newton}),   $\mappingPD_k\ilparenthesis{\cdot}$ is a mapping  from $\real^{\Dimension\times\Dimension}$ to the set of symmetric and positive definite matrices. 
	The straightforward form for $\mappingPD_k$ is the approximate damped Newton (see Remark~\ref{rem:damping}): $
	\mappingPD_k\ilparenthesis{\obH_k} = \obH_k + \correction_k \bI$,  for $ \correction_k> \uplambda_{\min}\ilparenthesis{\obH_k}$. 	Another popular mapping is $ \mappingPD_k\ilparenthesis{\obH_k} = \ilparenthesis{\obH_k ^\transpose\obH_k + \correction_k \bI}^{\nicefrac{1}{2}} $ for $\correction_k \to 0$, where the square root here is the unique positive definite square root (implementable via \texttt{sqrtm} in \textsc{MatLab}). 
 
 For succinctness, write  $\lossNoisy_k^\pm \equiv  \lossNoisy\ilparenthesis{\hbtheta_k\pm \perturb_k \direction_k  , \upomega_k^\pm }   $ and $ \lossNoisy_k \equiv  \lossNoisy\ilparenthesis{\hbtheta_k,\upomega_k } $. Write $ \lossNoise_{k}^{\pm} \equiv \lossNoisy_k^\pm - \loss\ilparenthesis{\hbtheta_k\pm \perturb_k\direction_k } $ and $ \lossNoise _k \equiv    \lossNoisy \ilparenthesis{\hbtheta_k  ,\upomega_{k}  } - \loss\ilparenthesis{\hbtheta_k  }   $.   
Since only noisy evaluation of the loss function is accessible, the gradient estimate  can be as  (\ref{eq:ghat}) using one or two noisy ZO oracles
\begin{subnumcases}{\label{eq:ghat}  	\hbg_k  =}
 \perturb_k^{-1} \lossNoisy_k^+    \direction_k   \,,   \label{eq:ghat-1} \\
  \perturb_k^{-1} \ilparenthesis{ \lossNoisy_k^+ - \lossNoisy_k  }   \direction_k   \,,   \label{eq:ghat-2} \\
 \ilparenthesis{2  \perturb_k}^{-1} \ilparenthesis{ \lossNoisy_k^+ - \lossNoisy_k^- }   \direction_k \,, \label{eq:ghat-3}
\end{subnumcases} 
and the Hessian estimate can be as (\ref{eq:Hhat}) using one to three noisy ZO oracles 
\begin{subnumcases}{\label{eq:Hhat}  	\hbH_k   =} 
\perturb_k^{-2}  {\lossNoisy_k^+ }    \ilparenthesis{\direction_k\direction_k^\transpose-\bI}   \,,   \label{eq:Hhat-0} \\
\perturb_k^{-2} \ilparenthesis{\lossNoisy_k^+ - \lossNoisy_k  }    \ilparenthesis{\direction_k\direction_k^\transpose-\bI}   \,,   \label{eq:Hhat-1} \\
\ilparenthesis{2  \perturb_k^2}^{-1} \ilparenthesis{ \lossNoisy_k^+ + \lossNoisy_k^-  }     \ilparenthesis{\direction_k\direction_k^\transpose-\bI}   \,, \label{eq:Hhat-2} \\
\ilparenthesis{2  \perturb_k^2}^{-1} \ilparenthesis{  \lossNoisy_k^+ + \lossNoisy_k^- - 2 \lossNoisy_k  }     \ilparenthesis{\direction_k\direction_k^\transpose-\bI}   \,. \quad \quad \label{eq:Hhat-3} 
\end{subnumcases}  The Hessian estimators (\ref{eq:Hhat}) are  guaranteed to be symmetric.  {This contrasts with \algoName{2SPSA}  \citet{spall2000adaptive}  where manual symmetrization is required.}

We focus on the uncontrolled\footnote{For the controllable noise scenario where $\upomega_k^{\pm } = \upomega_k$, frequently encountered in simulation optimization or the fixed-dataset discussed in Sect.~\ref{sect:data},  $ O(k^{-\nicefrac{1}{2}}) $ in terms of   RMS.   } noise scenario where $\upomega_k^\pm, \upomega_k$ are i.i.d., frequently encountered in datastream discussed in Sect.~\ref{sect:data}, $ O(k^{-\nicefrac{1}{3}}) $ in terms of root mean squared (RMS) error $ \ilparenthesis{\E\ilbracket{\norm{\hbtheta_k-\bvartheta}^2}}^{\nicefrac{1}{2}} $ is the \emph{fastest} rate possible for $\gain_k = O(k^{-1})$ and $ \perturb_k = O(k^{-\nicefrac{1}{6}}) $,  when $\loss\ilparenthesis{\cdot}$ is thrice continuously differentiable and is not quadratic.

\subsection{Implementation Aspects}

	The entire second-order algorithms (\ref{eq:Newton}) is summarized in 
	Algorithm~\ref{algo:Newton}.  
	
	  \begin{algorithm}[!htbp]
	 	\caption{Zeroth-Order Damped Hessian Algorithm Using Normal Perturbation } 
	 	\begin{algorithmic}[1]  
	 		\renewcommand{\algorithmicrequire}{\textbf{Input:}}
	 		\renewcommand{\algorithmicensure}{\textbf{Output:}}
	 		\Require  initialization $\hbtheta_0$, $\obH_0=\bI$; coefficients $ \gain_k, \gainH_k$ and $\perturb_k$ for $ 0\le k\le \iterTotal $. 
	 		\State \textbf{set} iteration counter $k=0$.  
	 		\State 	\For{$k=0,1,\cdots, \iterTotal$}{
	 			\textbf{generate} $\direction_k \sim\normalDist \ilparenthesis{\zero,\bI} $ and 
	 			\textbf{collect}   noisy loss  observations. 
	 			\newline   \;
	 			\;    \indent 
	 			\textbf{estimate} $\hbg_k$ via (\ref{eq:ghat}) and $ \hbH_k  $   via (\ref{eq:Hhat}). 
	 				\newline   \;
	 			\;    \indent 
	 			\textbf{update} $\hbtheta_{k+1}$ and $ \obH_{k+1} $ via (\ref{eq:Newton}).  
	 		}     
	 		\Ensure terminal estimate $ \hbtheta_{\iterTotal+1} $ (or iterate average). 
	 	\end{algorithmic}
	 	\label{algo:Newton}
	 \end{algorithm}

	\paragraph{$\hbtheta_0$ initialization}:  Following the last paragraph in Sect.~\ref{sect:damping}, 
	it is recommended to first run  the first-order scheme (\ref{eq:SA1st}),  and then switch to the second-order scheme (\ref{eq:SA2nd}) after $\hbtheta_k$ has reached the vicinity of $\bvartheta$ where the loss function is \emph{locally quadratic}.

	\paragraph{Positive-definite mapping}:
	The crucial step before applying second-order estimation to the parameter update for $\hbtheta_k$ within an optimization context is enforcing the Hessian estimate to be positive-definite.  Note that the estimation errors in $\hbH_k$ \emph{may} result in negative eigenvalues of $\obH_k$, which  is inevitable due to the noisy ZO oracles. 
	Generally, it costs $O(\Dimension^3)$ FLOPs  to guarantee a  symmetric matrix to be positive-definite \citep{zhu2002modified}.  \citep{zhu2019efficient} proposed to utilize the symmetric indefinite matrix decomposition to reduce the per-iteration FLOPs to $O(\Dimension^2)$.

  \paragraph{Gradient/Hessian averaging}: It is desirable to average several $\hbH_k$ and $\hbg_k $ values despite the additional query cost, especially in \emph{high-noise} environment.

		  \paragraph{Blocking}: We may enforce to set $\hbtheta_{k+1} = \hbtheta_k$ if the evaluation of the noisy evaluation at $\hbtheta_{k+1}$ is  substantially higher than that at $\hbtheta_k$ by a user-specified constant.

	\subsection{  Notation Conventions}\label{subsect:Notation}
	
	\paragraph{Matrix and vector operations} Let $\bA\in\real^{\Dimension\times\Dimension}$ be a matrix and $\bx\in\real^{\Dimension}$ be a vector. $ \norm{\bx} $ returns the Euclidean norm of $\bx$, and $\norm{\bA}$ returns the spectral norm of $\bA$.  If all eigenvalues of $\bA$ are real, $ \uplambda_{\min}\ilparenthesis{\bA} $ returns its smallest eigenvalue.  $ \bA\succ\zero (\bA\succeq\zero) $ means $\bA$ is symmetric and positive definite (semi-definite). 
	
	\paragraph{Probability} Let $\field_k$  represents the history\footnote{The precise definition for the filtration $\field_k$ may vary, depending on the estimator forms (\ref{eq:ghat}--\ref{eq:Hhat}) and the corresponding ZO queries. } of the process (\ref{eq:SA}). If $\hbtheta_0$ is random,   $\field_0$  is spanned by $\hbtheta_0$; if otherwise, $\field_0$ is trivial. 
	  Note that  $\hbtheta_{k+1}$ is a random variable that depends on the filtration $\field_k$ generated by the recursive algorithm (\ref{eq:SA})  before $\hbtheta_{k+1}$ is realized. Let $ \E_k\ilparenthesis{\cdot} $ denote the conditional expectation $ \E\ilbracket{\given{\cdot}{\field_k}} $.  ``Independent and identically distributed" is abbreviated as i.i.d. ``Infinitely often'' is abbreviated as i.o. ``Almost surely'' is abbreviated as a.s. The equality of two random variables almost surely \emph{may} be written as equality for clarity.

	\paragraph{Miscellaneous}  The binary operator $\otimes$ denotes the Kronecker product. 
	In addition to    $ \bg\ilparenthesis{\btheta}\equiv \nabla \loss\ilparenthesis{\btheta} \in\real^{\Dimension\times 1}$ and $ \bH\ilparenthesis{\btheta}\equiv \nabla^2\loss\ilparenthesis{\btheta} \in\real^{\Dimension\times\Dimension} $, we also  let $ \nabla^3\loss\ilparenthesis{\btheta} \in\real^{1\times \Dimension^3} $ (as a \emph{row} vector) be the third-order derivative of the loss function $\loss\ilparenthesis{\cdot}$.

	\section{Convergence Theory }\label{sect:Theory}

	For clarity and for  Remarks~\ref{rem:gAccuracy}--\ref{rem:HAccuracy} (to appear), we will analyze  $\hbg_k$ in  (\ref{eq:ghat-3}) (using two ZO queries)  and $\hbH_k$ as (\ref{eq:Hhat-3})  (using three ZO queries) subsequently.

	\subsection{Model Assumptions}

	We present conditions under which $\hbtheta_k\to\bvartheta$ and $\obH_k\to\bH\ilparenthesis{\bvartheta}$ almost surely as $k\to\infty$.  
	Write  the ``conditioned'' gradient estimate as  $\obg_k  \equiv    \ilbracket{ \mappingPD_k\ilparenthesis{\obH_k}}^{-1} \bg\ilparenthesis{\hbtheta_k} $.

	\begin{assumeA}[Conditions on $\loss\ilparenthesis{\cdot}$] \label{assume:loss3Lips}
		The loss function $\loss\ilparenthesis{\cdot}:\real^{\Dimension}\mapsto\real$ is thrice continuous differentiable, and $\nabla^3\loss\ilparenthesis{\cdot}$ is $\LipsPara_3$-Lipschitz continuous\footnote{It means that $ \norm{\nabla^3\loss\ilparenthesis{\btheta} - \nabla^3\loss\ilparenthesis{\bzeta}} \le \LipsPara_3 \norm{\btheta-\bzeta} $, where the third-order derivative $ \nabla^3\loss\ilparenthesis{\cdot} $ and the Euclidean norm $\norm{\cdot}$ were defined in Subsection~\ref{subsect:Notation}.  } and bounded.

	\end{assumeA}

\remove{ 
	\begin{assumeAprime}{1'}[Conditions on $\loss\ilparenthesis{\cdot}$] \label{assume:loss4Bound}
 	The loss function $\loss\ilparenthesis{\cdot}:\real^{\Dimension}\mapsto\real$ is four-times continuous differentiable, and $\nabla^4\loss\ilparenthesis{\cdot}$ is bounded for all $\btheta$. 
\end{assumeAprime}
}

	\begin{assumeA}
		[Noisy ZO Queries] \label{assume:Noise} For  the gradient-estimator (\ref{eq:ghat-3}), assume     $ \E\ilbracket{\given{ \lossNoise_k^+-\lossNoise_k^-  }{\hbtheta_k, \direction_k}} = 0 $  and 	$ \E\ilbracket{\given{\ilparenthesis{\lossNoise_k^+ - \lossNoise_k^-}^2}{\hbtheta_k,\direction_k}} $ is uniformly bounded for all $k$.  	For the  Hessian-estimator (\ref{eq:Hhat-3}),  assume 
		  $ \E\ilbracket{\given{ \lossNoise_k^+ + \lossNoise_k^- - 2\lossNoise_k }{\hbtheta_k, \direction_k}} = 0 $ and $ \E\ilbracket{\given{\ilparenthesis{\lossNoise_k^+ + \lossNoise_k^- - 2 \lossNoise_k}^2}{\hbtheta_k,\direction_k}} $  is uniformly bounded for all $k$.  
	\end{assumeA}

	\begin{assumeA}[Random Perturbation] \label{assume:Perturbation}
		$ \direction_k\stackrel{\mathrm{i.i.d.}}{\sim} \normalDist\ilparenthesis{\zero,\bI} $ and $\direction_k$ is independent of $\field_k$. Moreover,  both $ \E\ilset{\ilbracket{ \loss\ilparenthesis{\hbtheta_k\pm\perturb_k\direction_k} }^2}$ and  $ \E\ilset{\ilbracket{ \loss\ilparenthesis{\hbtheta_k } }^2}$  are uniformly bounded. 
	\end{assumeA}

\begin{assumeA}
	[Stepsizes] \label{assume:Stepsize} The positive  stepsizes satisfy $ \gain_k, \gainH_k,  \perturb_k \stackrel{k\to\infty}{\longrightarrow} 0 $, $ \sum_k\gain_k=\infty $,   $ \sum_k \ilparenthesis{\nicefrac{\gain_k}{\perturb_k}}^2<\infty $, and $ \sum_k \parenthesis{\nicefrac{\gainH_k}{\perturb_k^2}}^2<\infty  $. 
\end{assumeA}

\begin{assumeA} \label{assume:UniformlyBounded} 
	$ \limsup_{k\to\infty} \norm{\hbtheta_k} <\infty $  a.s. Moreover,  for any $\uprho>0$ and for all $i\in \set{1,\cdots,\Dimension}$, 
	\begin{equation}\label{eq:Bounded}
 {\text{  $ 	\Prob\parenthesis{E_k\,\, \mathrm{ i.o. }} = 0 \text{ where  the  event }E_k \equiv  
		\set{  \sign \ilparenthesis{ \overline{g}_{k,i} } \neq \sign \ilparenthesis{g_i\ilparenthesis{\hbtheta_k}} \text{ and } |  \ilparenthesis{\hbtheta_k}_i- \ilparenthesis{\bvartheta}_i |   \ge \uprho
		}\,.  $  }}
	\end{equation}
\end{assumeA}

	\begin{assumeA}
		[Sufficient Curvature] \label{assume:Curvature} 	For each $k$ and for all $\btheta$, there exists a $\convexPara>0$ such that $ \ilparenthesis{\hbtheta_k -\bvartheta}^\transpose \obg_k  \ge \convexPara\norm{\hbtheta_k-\bvartheta} $. 
	\end{assumeA}

	\begin{assumeA}
		[Conditions for $\mappingPD_k\ilparenthesis{\cdot}$] \label{assume:PDmapping}  
			  $ \perturb_k^2\ilbracket{ \mappingPD_k\ilparenthesis{\obH_k}}^{-1} \to \zero $ almost surely as $k\to\infty$.
			  For some $\BoundHessianMoment>0$, $ \E\ilbracket{\norm{\ilparenthesis{ \mappingPD_k\ilparenthesis{\hbH_k}}^{-1} }^{2+\BoundHessianMoment} } $ is bounded uniformly for all $k$. 
			  $ \norm{\mappingPD_k\ilparenthesis{\obH_k} - \obH_k} \to 0 $ almost surely as $k\to\infty$.  
	\end{assumeA}

\paragraph{Comments on assumptions} A.\ref{assume:Noise}--A.\ref{assume:Perturbation} are standard in SA algorithms. $ \sum_k\ilparenthesis{\nicefrac{\gainH_k}{\perturb_k^2}}^2<\infty $ in A.\ref{assume:Stepsize} is to ensure the bounded variance of the Hessian estimator.  
\citet[pp. 40--41]{kushner2012stochastic}   
 explains why ``$ \lim\sup_{k\to\infty} \norm{\hbtheta_k}<\infty $ a.s.'' in A.\ref{assume:UniformlyBounded}   is not a stringent  condition and could
 be expected to hold in most applications. (\ref{eq:Bounded})  in A.\ref{assume:UniformlyBounded}  ensures that $\hbtheta_k$ cannot be bouncing around to cause the signs of the conditioned gradient elements to be changing infinitely often when $\hbtheta_k$ is strictly bounded away from $\bvartheta$. 
 A.\ref{assume:loss3Lips} and 
 A.\ref{assume:Curvature} ensure the smoothness\footnote{Note that \algoName{2SPSA} requires four-times continuously differentiable with bounded fourth-order derivatives. We change it to  thrice continuously differentiable with Lipschitz-continuous third-order derivatives. } and steepness of $\loss\ilparenthesis{\cdot}$.   Note that experimenter has full control over 
 the stepsize and the mapping, A.\ref{assume:Stepsize} and A.\ref{assume:PDmapping} can be met easily. 

\subsection{Asymptotic  Unbiasedness}
 {Let us  leverage Lemma~\ref{lem:Stein} to construct an asymptotically unbiased gradient and Hessian estimators. This is instrumental in establishing convergence and achieving optimal convergence rate later on. }
\begin{lem}
	[Bias/Variance in Gradient Estimate] \label{lem:BiasGradient}  Under A.\ref{assume:loss3Lips}--A.\ref{assume:Perturbation}, the gradient estimators (\ref{eq:ghat})  satisfy  $  \E_k\ilparenthesis{\hbg_k }       \stackrel{\mathrm{a.s.}}{=} \bg\ilparenthesis{\hbtheta_k}+ O(\perturb_k^2) $. Furthermore,  $ \E _k\ilparenthesis{ \norm{\hbg_k}^2} \stackrel{\mathrm{a.s.}}{=} O(\perturb_k^{-2}) $.  
\end{lem} 
\begin{rem}\label{rem:gAccuracy}
Despite Lemma~\ref{lem:BiasGradient} applies to all three estimators in (\ref{eq:ghat}), it is recommended to use (\ref{eq:ghat-3}) over (\ref{eq:ghat-1}--\ref{eq:ghat-2}). The rationale behind it is discussed in Appendix~\ref{appendix}\remove{, where  the conditions for (\ref{eq:BiasGradient-5}) is much less stringent than those for  (\ref{eq:BiasGradient-6}--\ref{eq:BiasGradient-7})}. 
\end{rem}
	
	\begin{lem}[Bias/Variance in Hessian Estimate]
	\label{lem:BiasHessian}	Under A.\ref{assume:loss3Lips}--A.\ref{assume:Perturbation}, the Hessian estimator (\ref{eq:Hhat})  satisfy $  \E_k\ilparenthesis{\hbH_k }       \stackrel{\mathrm{a.s.}}{=} \bH\ilparenthesis{\hbtheta_k}+ O(\perturb_k^2) $ and  $\E_k\ilparenthesis{\norm{\hbH_k}^2}\stackrel{\mathrm{a.s.}}{=}  O(\perturb_k^{-4})$. 
	\end{lem}

	\begin{rem}\label{rem:HAccuracy}
		For the same reason in   Remark~\ref{rem:gAccuracy}, it is recommended to use (\ref{eq:Hhat-3}) over (\ref{eq:Hhat-0}--\ref{eq:Hhat-2}). 
	\end{rem}

	\remove{ 
	
	\begin{lemPrime}{2'} [Bias in Hessian Estimate] 
		\label{lem:BiasHessianPrime}	Under A.\ref{assume:loss4Bound},  A.\ref{assume:Noise}, and  A.\ref{assume:Perturbation}, $ \norm{ \E \ilbracket{\given{\hbH_k}{\field_k}   } - \bH\ilparenthesis{\hbtheta_k} } _{\infty} = O(\perturb_k^2) $. 
	\end{lemPrime}

}

	\subsection{Almost Sure Convergence}
	 {With the gain sequence properly weighting the bias and variance in the gradient and Hessian estimators, we can establish the almost surely convergence $ \hbtheta_k\stackrel{\mathrm{a.s.}}{\longrightarrow}\bvartheta $ as $k\to\infty$.  }
	\begin{thm}
		[Strong Convergence of Parameter]  \label{thm:ParameterConvergence} Under A.\ref{assume:loss3Lips}--A.\ref{assume:PDmapping}, $\hbtheta_k\to\bvartheta$ a.s.    as $k\to\infty$. 
	\end{thm}

 \begin{thm}
 	[Strong Convergence of Hessian]  \label{thm:HessianConvergence} Under A.\ref{assume:loss3Lips}--A.\ref{assume:PDmapping}, $  \obH_k\to\bH\ilparenthesis{\bvartheta}$ a.s.    as $k\to\infty$, where $\obH_k$ is governed by (\ref{eq:Newton-2}) and $\hbH_k$ is computed per (\ref{eq:Hhat-3}). 
 \end{thm}

\subsection{Rate of Convergence }

 Let us find the convergence rate here. We enforce the form of the stepsizes: $ \gain_k = \nicefrac{\gain}{(k+1)^{\upalpha}} $, $ \perturb_k=\nicefrac{\perturb}{(k+1)^{\upgamma}} $, and $ \gainH_k $ satisfies A.\ref{assume:Stepsize}. Let  $
 \uptau\equiv \upalpha-2\upgamma$ and $ \uptau_+ \equiv \uptau\cdot\indicator_{\ilset{\upalpha=1}}$

\begin{thm}
	[Asymptotic Normality]  \label{thm:Normality} Assume A.\ref{assume:loss3Lips}--A.\ref{assume:PDmapping} hold. Pick $\gain> \nicefrac{\uptau_+ }{\ilbracket{ 2\uplambda_{\min} \ilset{        \bH\ilparenthesis{\bvartheta}    } }} $ and $ \upalpha\le 6\upgamma $, we have 
	\begin{equation}\label{eq:normal}
	k^{\nicefrac{\uptau}{2}} \ilparenthesis{\hbtheta_k-\bvartheta} \stackrel{\mathrm{dist.}}{\longrightarrow} \normalDist\ilparenthesis{\bmu,\bLambda}\,,
	\end{equation} 
	where      $\direction\sim\normalDist\ilparenthesis{\zero,\bI}$ and 
\begin{subnumcases}{\label{eq:normality} } \bmu =  \frac{\gain\perturb^2}{{3\uptau_+ -6 \gain}} \ilbracket{\bH\ilparenthesis{\bvartheta}}^{-1} \E\ilbracket{ \loss^{(3)}\ilparenthesis{\bvartheta}  \ilparenthesis{\direction\otimes\direction\otimes\direction}   
	\direction } \,, \label{eq:normMean} 
\\   \bLambda = \frac{\gain^2  \Var \ilbracket{\lossNoisy\ilparenthesis{\bvartheta,\upomega}}  }{2\perturb^2\ilparenthesis{2\gain-\uptau_+}}  \ilbracket{\bH\ilparenthesis{\bvartheta}}^{-2}\,.  \label{eq:normCov}
\end{subnumcases}

\end{thm}

\begin{rem}
	Note that when $\upomega\sim\Prob$ is drawn i.i.d. from $\Omega$, we have $ \E\ilbracket{\given{\ilparenthesis{\lossNoise_k^+ - \lossNoise_k^-}^2}{\hbtheta_k,\direction_k}} \to 2\Var\ilbracket{\lossNoisy\ilparenthesis{\bvartheta,\upomega}} $ a.s., where the variance is taken over $\upomega\in\Omega$.  This is due to $\hbtheta_k\stackrel{\mathrm{a.s.}}{\longrightarrow} \bvartheta$ shown in Theorem~\ref{thm:ParameterConvergence} and $\perturb_k\to 0$ assumed in A.\ref{assume:Stepsize}. 
\end{rem}

	\begin{rem}
		Although the possible fastest rate  remains to be $ O(k^{-\nicefrac{1}{3}}) $ for $\gain_k = O(k^{-1})$,  $ \perturb_k = O(k^{-\nicefrac{1}{6}})$,  and $ \gainH_k $ such that $ \sum_k\gainH_k^2 \perturb_k^{-4}<\infty $, 
		 (\ref{eq:Hhat-3}) only requires \emph{three} ZO queries to form one single Hessian estimate  whereas \algoName{2SPSA} requires \emph{four} ZO queries. Moreover, (\ref{eq:Hhat-3}) only requires generating  \emph{one} perturbation $\direction_k$  (via Monte Carlo) and differencing stepsize $\perturb_k$ (which requires hyper-parameter tuning), while \algoName{2SPSA} requires \emph{two} statistically independent perturbations and  \emph{two} differencing stepsizes \citep{zhu2020sotchastic}.

	\end{rem}

\subsection{{Comparison with \algoName{2SPSA}}} \label{sect:compare}
We can now explain Sect.~\ref{sect:contribution} in more details.	Specifically, \algoName{2SPSA} for per-iteration Hessian estimate requires sampling two random perturbations $ \bDelta_k $ and $ \tilde{\bDelta}_k $ (commonly each component of the random perturbations are i.i.d. Rademacher distributed) and tuning two perturbation magnitudes $ \perturb_k $ and $ \tperturb_k $. It has the following form:
\begin{equation}\label{eq:2SPSA}
\begin{cases}
&\hbH_k \gets \ilparenthesis{2\perturb_k\tperturb_k} ^{-1} \ilparenthesis{\lossNoisy_k^{+,+}  - \lossNoisy_k^{-,+}  + \lossNoisy_k^-- \lossNoisy_k^+} \ilparenthesis{ \widetilde{\bDelta}_k^{-1} \bDelta_k^{-\transpose} }\,,\\
&\hbH_k \gets \nicefrac{1}{2} \ilparenthesis{\hbH_k + \hbH_k^\transpose}\,, \quad\text{ for symmetrization}\,, 
\end{cases}
\end{equation}
	where $ \lossNoisy_k^{\pm} = \lossNoisy\ilparenthesis{\hbtheta_k\pm\perturb_k\bDelta_k, \upomega_k^{\pm}} $ and $ \lossNoisy_k^{\pm, +} = \lossNoisy(\hbtheta_k \pm \perturb_k\bDelta_k + \tperturb_k\widetilde{\bDelta}_k, \upomega_k^{\pm, +} ) $, $ \bDelta^{-1} $ indicates component-wise inverse\footnote{\citet{spall2000adaptive} requires that both $\bDelta_k$ and $\tilde{\bDelta}_k$ have bounded inverse-second-moment. For Rademacher distributed random perturbation, component-wise inverse is well-defined as every component has zero probability of being zero.}. 
	Comparing (\ref{eq:2SPSA}) and (\ref{eq:Hhat-3}) side-by-side, we shall see the bullet-point comparison in  Sect.~\ref{sect:contribution}, which is further summarized in Table~\ref{tab:summary}. 
\begin{table}[!htbp]
	\begin{tabular}{|l||c|c|c|c|c|} \hline
		 	& Optimal Rate*  & Queries & Injected Randomness &Hyper-para & Weighting $ w_k $ \\ \hline\hline
		\algoName{2SPSA}	& \multirow{2}{*}{$ O(k^{-\nicefrac{1}{3}}) $}  & 4  & Rademacher $\bDelta_k$, $\widetilde{\bDelta}_k$ &$ \perturb_k $, $ \tperturb_k $& $ \nicefrac{1}{k+1} $  \\  
		Algo~\ref{algo:Newton}	&  &3   & Standard MVN $ \direction_k $ &$\perturb_k$ & $ \sum_k {w_k}{\perturb_k^{-4}}<\infty $ \\ \hline 
	\end{tabular}
	\caption{Comparison between \algoName{2SPSA} and Algorithm~\ref{algo:Newton}. }
	\label{tab:summary}
\end{table}

Note that the convergence rate  * in \algoName{2SPSA} requires ``four-times continuously differentiable with bounded fourth-order derivatives'', while Algorithm~\ref{algo:Newton} requires ``thrice continuously differentiable with Lipschitz-continuous third-order derivatives.'' Concretely, the bias and variance for the gradient/Hessian estimates for \algoName{2SPSA} requires the bound of the fourth-order derivative, while those in Lemma~\ref{lem:BiasGradient}/\ref{lem:BiasHessian} requires the Lipschitz continuity constant of the third-order derivatives. Moreover, the fourth-order continuity in \algoName{2SPSA} cannot be relaxed, as it is required in  Taylor-expansion derivation.

To further dissect the big-$O$   optimal rate  in Table~\ref{tab:summary}, we interpret Theorem~\ref{thm:Normality} following \citet[Sect. 3]{zhu2020harp}: $
\lim_{k\to\infty} \ilbracket{\E\ilparenthesis{\norm{\hbtheta_k-\bvartheta}^2}}^{\nicefrac{1}{2}} = k^{-\nicefrac{\uptau}{2}} \bracket{\norm{\bmu}^2 + \tr\ilparenthesis{\bLambda}}$.  	The estimates generated by  \algoName{2SPSA} similarly satisfies $ k^{\nicefrac{\uptau}{2}} \ilparenthesis{\hbtheta_k^{\mathrm{2SPSA}}-\bvartheta} \stackrel{\mathrm{dist.}}{\longrightarrow}\normalDist\ilparenthesis{\bmu^{\mathrm{2SPSA}}, \bLambda^{\mathrm{2SPSA}}}$,  where $\bDelta$
 	and $ \bLambda^{\mathrm{2SPSA}}$ coincides with (\ref{eq:normCov}).   Let each component of $\bDelta$ be i.i.d. Rademacher distributed, then  $\bmu^{\mathrm{2SPSA}} = \nicefrac{\gain\perturb^2}{\ilparenthesis{3\uptau_+ - 6\gain}} \ilbracket{\bH\ilparenthesis{\bvartheta}}^{-1} \E\ilbracket{\loss^{(3)} \ilparenthesis{\bvartheta} \parenthesis{\bDelta\otimes\bDelta \otimes\bDelta} \bDelta^{-1}} $. Unfortunately, it is difficult to make a concrete comparison between $\bmu$ in (\ref{eq:normCov}) and $\bmu^{\mathrm{2SPSA}} $, as it depends on the third-order derivative of the underlying \emph{unknown} loss function $\loss\ilparenthesis{\cdot}$.

	\section{Numerical Experiment}\label{sect:Numerical}
	To support that Stein's identity helps reducing the per-iteration ZO-queries, this section provides one synthetic illustration and one real-data example. 
	\subsection{Skewed-Quartic Function}
	This section compares   Algorithm~\ref{algo:Newton} with the counterpart \algoName{2SPSA}. 
	The loss function here is the skew-quartic function
	\begin{equation*}
	\loss\ilparenthesis{\btheta} = \btheta^\transpose\bA^\transpose\bA\btheta + 0.1 \sum_{i=1}^{\Dimension} (\bA\btheta)_i^3 + 0.01  \sum_{i=1}^{\Dimension} (\bA\btheta)_i^4\,,
	\end{equation*}
	where $(\cdot)_i$ is the $i$th component of the argument vector, and $\bA$ is such that $\Dimension\bA$ is an upper-triangular matrix of all ones. The additive noise in $\lossNoisy\ilparenthesis{\cdot}$ is independent $ \normalDist\ilparenthesis{0,0.1} $, i.e., $ \lossNoisy\ilparenthesis{\btheta,\upomega} = \loss\ilparenthesis{\btheta} + \lossNoise\ilparenthesis{\upomega} $, where $\lossNoise\ilparenthesis{\upomega} \sim\normalDist\ilparenthesis{0,0.1}$. Though $ \bH\ilparenthesis{\btheta}\succ \nicefrac{5}{4}\bA^\transpose\bA $, the Hessian function is   ill-conditioned\footnote{{For example, when $\Dimension=20$,  the Hessian matrix evaluated at a vector of all ones is approximately $778$.} }.

	We use $\Dimension=20$ and initialize $\hbtheta_0$ within $ \ilbracket{-5,5}^{\Dimension} $. Both algorithms are run with iteration $\iterTotal=10000$ with $50$ replicates (i.e., all the performance will be averaged over $50$ replicates). 
	We use  $\perturb_k=\nicefrac{1}{(k+1)^{0.101}}$ and $ \gain_k = \nicefrac{1}{(k+1+A)}^{0.602} $  and $A$ equals $ 10\% $ of the total iteration number $\iterTotal$ for both (\ref{eq:Newton}) and \algoName{2SPSA}. For fair comparison, we use $ \gainH_k = \nicefrac{1}{k+1} $ for both methods. During the implementation, both algorithms use exactly twelve\footnote{Using twelve queries, four $\hbH_k$ can be produced for Algorithm~\ref{algo:Newton}, and three $\hbH_k$ can be produced for \algoName{2SPSA}.  } ZO queries per iteration, so that the query complexity \emph{aligns} with the iteration complexity. We see from Figure~\ref{fig:Skew-Quartic} that Algorithm~\ref{algo:Newton} is more query-efficient than \algoName{2SPSA}, in terms of the attainable accuracy using the \emph{same} ZO queries. 
	
	\begin{figure}[!htbp]
			\centering
	\includegraphics[width=.5\linewidth]{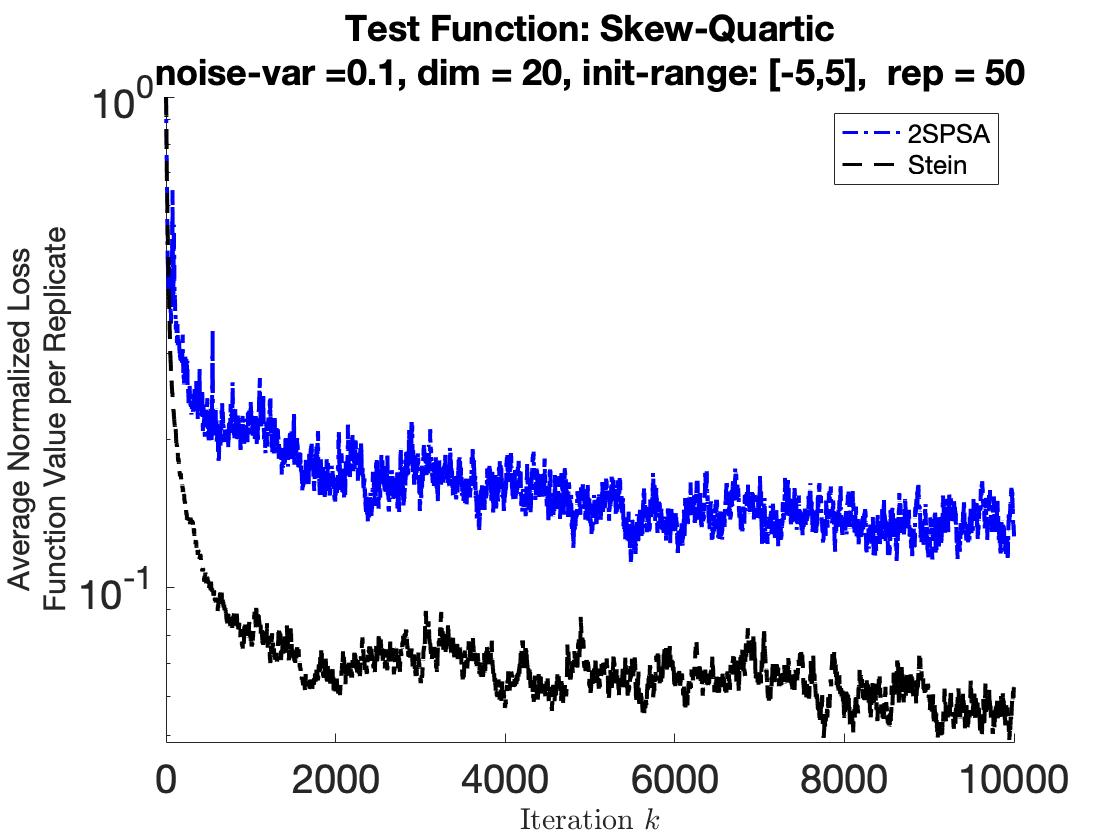} 
		\caption{  Performance of Algorithm~\ref{algo:Newton} and \algoName{2SPSA} in terms of normalized distance $ \nicefrac{\ilbracket{\loss\ilparenthesis{\hbtheta_k} - \loss \ilparenthesis{\bvartheta}}}{\ilbracket{\loss\ilparenthesis{\hbtheta_0} - \loss \ilparenthesis{\bvartheta}}} $  average  across $50$ independent replicates. Both algorithms use twelve ZO queries per iteration,  {so query complexity aligns with iteration complexity}. The underlying loss function is the skew-quartic function with $\Dimension=20$, and the noisy observation is corrupted by $\normalDist\ilparenthesis{0, 0.1}$ noise.    }
		\label{fig:Skew-Quartic}
	\end{figure}

	\subsection{ {Black-Box Binary Classification}}
	\label{sect:classification} 
	
 	We now apply both \algoName{2SPSA} and the proposed algorithm to solve the black-box binary classification task \citet{huang2020accelerated}. The dataset  \href{https://www.csie.ntu.edu.tw/~cjlin/libsvmtools/datasets/binary.html#phishing}{\textsc{Phishing}}   has $\sampleTotal = 11055$ samples $ \ilparenthesis{\bx_\sampleIndex, y_\sampleIndex}_{1\le \sampleIndex \le \sampleTotal} $, where the features $\bx_\sampleIndex$ are $68$-dimensional, and $ y_\sampleIndex\in\set{-1,1} $. Consider the following nonconvex corr-entropy induced loss  \begin{equation}\label{eq:classification}
	 \loss\ilparenthesis{\btheta} = \frac{1}{\sampleTotal} \sum_{\sampleIndex=1}^I \frac{\upkappa ^2}{2} \set{1- \exp\bracket{-\frac{\ilparenthesis{y_\sampleIndex-\bx_\sampleIndex^\transpose\btheta}^2}{\upkappa^2}} }\,,
	 \end{equation} 
	 where  $\btheta$ is  $68$-dimensional, and the $\upkappa$ is some penalty parameter. Note that $\loss\ilparenthesis{\btheta}=0$ when all samples are correctly classified. The noisy observation $\lossNoisy\ilparenthesis{\btheta,\upomega}$ is 
	 \begin{equation}\label{eq:classification2}
	 \lossNoisy \ilparenthesis{\btheta,\upomega} = \frac{1}{\sampleBatch} \sum_{j=1}^\sampleBatch \frac{\upkappa ^2}{2} \set{1- \exp\bracket{-\frac{\ilparenthesis{y_{i_j\ilparenthesis{\upomega}}-\bx_{i_j\ilparenthesis{\upomega}}^\transpose\btheta}^2}{\upkappa^2}} }\,,
	 \end{equation}
	  for $J\le I$, and the $J$ indexes $ \ilset{i_1\ilparenthesis{\upomega},\cdots, i_J\ilparenthesis{\upomega}} $ are i.i.d. uniformly drawn from $ \set{1,\cdots,I} $ (without replacement as discussed in Sect.~\ref{sect:data}). 
	  
	  Consider (\ref{eq:classification}) with penalty parameter $\upkappa=10$ and (\ref{eq:classification2}) with mini-batch size $\sampleBatch=10$. Both \algoName{2SPSA} and Algorithm~\ref{algo:Newton} are initialized at 
	a 68-dimensional vector of all ones. The ZO-query per iteration for both algorithm is twelve\footnote{With $12$ queries, \algoName{2SPSA} can construct the average of \emph{three} Hessian estimators, as each of \algoName{2SPSA} Hessian estimators costs $4$ ZO-queries. Similarly, Algorithm~\ref{algo:Newton} can construct the average of \emph{four} Hessian estimators, each of which costs $3$ queries. }, so the query complexity \emph{aligns} with the iteration complexity. We perform $25$ independent replicates, each with $K=5000$ iterations. 	
	\begin{figure}[!htbp]
		\centering
		\includegraphics[width=.5\linewidth]{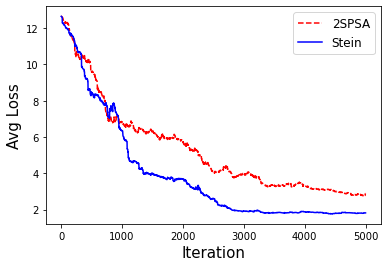} 
		\caption{  Performance of Algorithm~\ref{algo:Newton} and \algoName{2SPSA} in terms of the true loss function $\loss\ilparenthesis{\hbtheta_k}$  average  across $25$ independent replicates. Both algorithms use twelve ZO queries per iteration, so query complexity aligns with iteration complexity.  A zero loss function is equivalent to $100\%$ classification correctness. }
		\label{fig:Classification}
	\end{figure}
	Figure~\ref{fig:Classification} shows the convergence of the two algorithms on the black-box binary classification problem. We see that Algorithm~\ref{algo:Newton} is more query-efficient than \algoName{2SPSA} in terms of the attainable accuracy using the \emph{same} ZO queries.

	\section{Summary and Concluding Remarks}\label{sect:Concluding}
		
	This work proposes (\ref{algo:Newton}) in solving the minimization problem when noisy ZO oracle is available. 	The key improvement from (\ref{eq:Hhat-3}) to \algoName{2SPSA}
		 is the 
		 reduced per-iteration  ZO query number.
		Besides, ours and \algoName{2SPSA} differ in the number of functions sampled, the sampling scheme for random perturbations, and the way in which the gradient/Hessian approximations are derived, as summarized in Sect.~\ref{sect:contribution} and revisited in Sect.~\ref{sect:compare}.

		Several potential future work includes (i) extension to scenarios where unbiased direct measurements of $\bg\ilparenthesis{\btheta}$ are available; (ii) extension for possible distribution for  $\direction $, e.g.,  other continuous distributions per Remark~\ref{rem:distExtensionCont} and discrete distributions per Remark~\ref{rem:distExtensionDisc}; and (iii) leveraging importance sampling techniques  in Monte Carlo sampling our estimators once the distribution for perturbation $\direction$ is revealed. 
		
		\section*{Acknowledgment}
		The author would like to thank Dr. Zhenliang Zhang, Dr. Jian Tan, and Dr. Wotao Yin for inspirational discussion.

	\bibliography{ref}

	\appendix
	
	\section{Supplementary Proofs} \label{appendix}
	 
	\begin{proof}
		[Proof for Lemma~\ref{lem:Stein}] 
		
		\begin{enumerate}[i)]
			\item 
		Let us  discuss the   \emph{general} setting where the underlying loss function $\loss\ilparenthesis{\cdot}$   convolutes with a Gaussian distribution with a  covariance matrix\footnote{We assume all the covariance matrix are symmetric and positive-definite.} $\bSigma$:
		\begin{equation}\label{eq:LossGaussConv}
		\loss(\given{\btheta}{\bSigma})\equiv    \int_{\real^{\Dimension}} \loss (\bzeta ) \upvarphi(\given{\bzeta}{\btheta,\bSigma} )\diff \bzeta \,, 
		\end{equation}
		where  $
		\upvarphi  (\given{\bzeta}{\btheta,\bSigma})\equiv \ilbracket{  \ilparenthesis{2\uppi}^{\Dimension} \abs{\bSigma} }^{-\nicefrac{1}{2}} \exp \bracket{ - \nicefrac{1}{2} \ilparenthesis{\bzeta-\btheta}^\transpose\bSigma^{-1} \ilparenthesis{\bzeta - \btheta} }
	$
		is   the probability density function of the multivariate normal distribution $ \normalDist\ilparenthesis{\btheta,\bSigma} $. 
		By Leibniz's rule, 
		\begin{equation}\label{eq:Leibniz}
		\begin{cases}
	&	\nabla	\loss\ilparenthesis{\given{\btheta}{\bSigma}} = \int _{\real^{\Dimension}} \bSigma^{-1} \ilparenthesis{\bzeta-\btheta}    \upvarphi (\given{\bzeta}{\btheta,\bSigma}) \loss \ilparenthesis{\bzeta} \diff\bzeta \,,  \\
		&		\nabla^2	\loss\ilparenthesis{\given{\btheta}{\bSigma}} = \int_{\real^{\Dimension}}   \bracket{  \bSigma^{-1} \parenthesis{\bzeta-\btheta} \parenthesis{\bzeta-\btheta}^\transpose \bSigma^{-1}  -  \bSigma^{-1}  }   \upvarphi\ilparenthesis{\given{\bzeta}{\btheta,\bSigma}} \loss\ilparenthesis{\bzeta}\diff\bzeta \,.
		\end{cases}
		\end{equation}
		We now  rewrite (\ref{eq:LossGaussConv}--\ref{eq:Leibniz}) compactly as follows: 
		\begin{equation}\label{eq:SteinCompact}
		\begin{dcases}
		& 	\loss\ilparenthesis{\given{\btheta}{\bSigma}}  = \E\bracket{\loss \ilparenthesis{\btheta+ \bC \direction }}\,,  \\ 
		&  	\nabla \loss\ilparenthesis{\given{\btheta}{\bSigma}}  =  \E\bracket{ \bC^{-1} \direction    \loss \ilparenthesis{\btheta+\bC \direction }}\,, \\
		& 	\nabla^2\loss\ilparenthesis{\given{\btheta}{\bSigma}} =\E \set{\ilbracket{ \bC^{-1}\parenthesis{\direction\direction^\transpose -\bI  }\bC^{-1}}    \loss \ilparenthesis{\btheta+\bC \direction}}\,, 
		\end{dcases}
		\end{equation}
		where the expectation is taken with respect to the random variable  $\direction\sim\normalDist\ilparenthesis{\zero, \bI}$ and $\bC$ is a symmetric matrix\footnote{Suppose  $\bSigma$ has eigen-decomposition form as $ \bU\bD\bU^\transpose $ for some unitary matrix $\bU$ and diagonal matrix $\bD$, then $\bC = \bU \sqrt{\bD}\bU^\transpose$ is symmetric and satisfies $ \bSigma=\bC^2 $. } such that $ \bSigma= \bC ^2 $. 	  
		Recall  that $\loss_{\perturb} \ilparenthesis{ {\btheta}{ }} = \E_{\direction\sim\normalDist\ilparenthesis{\zero,\bI}}\bracket{\loss\ilparenthesis{\btheta+\perturb\direction}}$. Letting $\bC = \perturb\bI$ (equivalently  $\bSigma= \perturb^2\bI$) in 
		(\ref{eq:SteinCompact}) gives  (\ref{eq:SteinGradient}--\ref{eq:SteinHessian}), which  is a special case of  Proposition~\ref{prop:Stein}.
		
		\item (\ref{eq:SteinGradient2}) can be  obtained 
		using $ \E(\direction)=\zero $ 
		and  $ \nabla \loss_{\perturb}\ilparenthesis{\btheta} = - \E  \bracket{\perturb^{-1}  \direction \loss\ilparenthesis{\btheta-\perturb\direction}    } $. 	
		(\ref{eq:SteinHessian2}) can be obtained using $ \E \ilparenthesis{\direction\direction^\transpose} = \bI  $ 
		and  $ \nabla^2\loss_{\perturb}\ilparenthesis{\btheta} = \E \bracket{\perturb^{-2}   \parenthesis{\direction\direction^\transpose-\bI} \loss\ilparenthesis{\btheta-\perturb\direction}   }  $. 	
		\item   (\ref{eq:SteinHessian3}) can be  directly derived from  (\ref{eq:SteinHessian2}).  
	\end{enumerate} 
		\end{proof}

	\begin{proof}
		[Proof of Lemma~\ref{lem:BiasGradient}]   
		\begin{itemize}
			\item \textbf{Bias for (\ref{eq:ghat-3})}	 
			\begin{align}
			& 	\E_k \ilparenthesis{\hbg_k} \stackrel{ \text{(\ref{eq:ghat-3})} }{=} (2\perturb_k)^{-1} \E_k\set{ \ilbracket{ \loss\ilparenthesis{\hbtheta_k+\perturb_k\direction_k} + \lossNoise_k^+ - \loss\ilparenthesis{\hbtheta_k - \perturb_k\direction_k} - \lossNoise_k^- } \direction_k} \nonumber\\
			&\stackrel{\text{a.s.}}{=} (2\perturb_k)^{-1} \parenthesis{\E_k  \ilset{ \ilbracket{ \loss\ilparenthesis{\hbtheta_k+\perturb_k\direction_k}   - \loss\ilparenthesis{\hbtheta_k - \perturb_k\direction_k}  } \direction_k} + \E_k \ilbracket{ \direction_k  \E\ilparenthesis{\given{\lossNoise_k^+ - \lossNoise_k^-}{\hbtheta_k,\direction_k}}  } }  \label{eq:BiasGradient-2} \\
			&\stackrel{\text{a.s.}}{=} (2\perturb_k)^{-1} \E_k \parenthesis{\E \set{\given{  \ilbracket{\loss\ilparenthesis{\hbtheta_k+\perturb_k\direction_k} - \loss\ilparenthesis{\hbtheta_k-\perturb_k}  }\direction_k }{\hbtheta_k}}} \label{eq:BiasGradient-3} \\
			&\stackrel{\text{a.s.}}{=} \nabla\loss_{\perturb_k}\ilparenthesis{\hbtheta_k}\,, \label{eq:BiasGradient-4}
			\end{align}
					where (\ref{eq:BiasGradient-2}) is due to \citet[Thm. 9.1.3 on p. 315]{chung2001course}, (\ref{eq:BiasGradient-3}) is valid given A.\ref{assume:Noise}, and (\ref{eq:BiasGradient-4}) is due to (\ref{eq:SteinGradient2-2}). When A.\ref{assume:loss3Lips} holds, we have  $ |  \loss\ilparenthesis{\bzeta} - \loss\ilparenthesis{\btheta} - \ilparenthesis{\bzeta-\btheta}^\transpose\bg\ilparenthesis{\btheta}  -2^{-1}\ilparenthesis{\bzeta-\btheta}^\transpose\bH\ilparenthesis{\btheta}\ilparenthesis{\bzeta-\btheta}   |  = O(\norm{\bzeta-\btheta}^3) $ by Taylor expansion.  Moreover, 
			\begin{align}
			2\perturb\ilbracket{\nabla \loss_{\perturb}\ilparenthesis{\btheta} - \bg\ilparenthesis{\btheta}} 
			&\stackrel{\text{(\ref{eq:SteinGradient2-2})}}{=} \E_{\direction} \ilset{\ilbracket{\loss\ilparenthesis{\btheta+\perturb\direction} - \loss\ilparenthesis{\btheta-\perturb\direction}} {\direction}} - 2\perturb\bg\ilparenthesis{\btheta} \nonumber\\
			&\,=\E \set{\ilbracket{ \loss\ilparenthesis{\btheta+\perturb\direction} - \loss\ilparenthesis{\btheta }  - \perturb\direction^\transpose\bg\ilparenthesis{\btheta}  - 2^{-1} \perturb^2  \direction^\transpose\bH\ilparenthesis{\btheta} \direction     } \direction   }    \nonumber\\
			&\quad  -\E  \set{ \ilbracket{\loss\ilparenthesis{\btheta-\perturb\direction }  -\loss\ilparenthesis{\btheta } +  \perturb\direction^\transpose\bg\ilparenthesis{\btheta}        - 2^{-1} \perturb^2  \direction^\transpose\bH\ilparenthesis{\btheta} \direction    } \direction  } \,, \label{eq:BiasGradient-5}
			\end{align} where (\ref{eq:BiasGradient-5}) uses    $ \E_{\direction} \ilbracket{ \direction^\transpose\bg\ilparenthesis{\btheta} \direction } = \bg\ilparenthesis{\btheta} $  for  $\direction\sim\normalDist\ilparenthesis{\zero,\bI}$ per   A.\ref{assume:Perturbation}.  Using the triangle inequality and the higher-moments for multivariate chi-squared distribution,   we have 
			\begin{equation}
			\norm{\nabla\loss_{\perturb}\ilparenthesis{\btheta} - \bg\ilparenthesis{\btheta}} \le \frac{\perturb^2O(1)}{6} \E_{\direction} \ilparenthesis{ \norm{\direction}^5 } =  \perturb^2 \times O(\Dimension^{\nicefrac{5}{2}})\,. \nonumber
			\end{equation}
			Combining above  with (\ref{eq:BiasGradient-4}), we have  $ \E_k\ilparenthesis{\hbg_k} \stackrel{\text{a.s.}}{=}\bg\ilparenthesis{\hbtheta_k} + O(\perturb_k^2) $ for (\ref{eq:ghat-3}).

			\item \textbf{Bias for (\ref{eq:ghat-2})}
Similarly, $ \E_k \ilparenthesis{\hbg_k} \stackrel{\text{a.s.}}{=} \nabla \loss_{\perturb_k}\ilparenthesis{\hbtheta_k}$	
 for estimator (\ref{eq:ghat-2}) thanks to (\ref{eq:SteinGradient2-1}). Moreover,
 \begin{align}
 \perturb\ilbracket{\nabla\loss_{\perturb} \ilparenthesis{\btheta} - \bg\ilparenthesis{\btheta}} &\stackrel{\text{(\ref{eq:SteinGradient2-1})}}{=} \E_{\direction} \ilset{\ilbracket{\loss\ilparenthesis{\btheta+\perturb\direction} - \loss\ilparenthesis{\btheta} }\direction} - \perturb\bg\ilparenthesis{\btheta} \nonumber\\
 &= \E \set{\ilbracket{ \loss\ilparenthesis{\btheta+\perturb\direction} - \loss\ilparenthesis{\btheta }  - \perturb\direction^\transpose\bg\ilparenthesis{\btheta}  - 2^{-1} \perturb^2  \direction^\transpose\bH\ilparenthesis{\btheta} \direction     } \direction   }  \,,  \label{eq:BiasGradient-6}
 \end{align}
 where (\ref{eq:BiasGradient-6}) uses   $ \E_{\direction} \ilbracket{ \direction^\transpose\bg\ilparenthesis{\btheta} \direction } = \bg\ilparenthesis{\btheta} $  and $ \E_{\direction}\ilset{ \ilbracket{\direction^\transpose\bH\ilparenthesis{\btheta}\direction} \direction} = \zero $ for $\direction\sim\normalDist\ilparenthesis{\zero,\bI}$. For estimator (\ref{eq:ghat-2}), $ \E_k\ilparenthesis{\hbg_k} \stackrel{\text{a.s.}}{=}\bg\ilparenthesis{\hbtheta_k} + O(\perturb_k^2) $ can be similarly derived.

 \item \textbf{Bias for (\ref{eq:ghat-1})}  Still, $ \E_k \ilparenthesis{\hbg_k} \stackrel{\text{a.s.}}{=} \nabla \loss_{\perturb_k}\ilparenthesis{\hbtheta_k}$	
 for estimator (\ref{eq:ghat-1}) thanks to (\ref{eq:SteinGradient}). Moreover,
 \begin{align}
 \perturb\ilbracket{\nabla\loss_{\perturb} \ilparenthesis{\btheta} - \bg\ilparenthesis{\btheta}} &\stackrel{\text{(\ref{eq:SteinGradient})}}{=} \E_{\direction} \ilset{ {\loss\ilparenthesis{\btheta+\perturb\direction}  }\direction} - \perturb\bg\ilparenthesis{\btheta} \nonumber\\
 &= \E \set{\ilbracket{ \loss\ilparenthesis{\btheta+\perturb\direction} - \loss\ilparenthesis{\btheta }  - \perturb\direction^\transpose\bg\ilparenthesis{\btheta}  - 2^{-1} \perturb^2  \direction^\transpose\bH\ilparenthesis{\btheta} \direction     } \direction   }  \,,  \label{eq:BiasGradient-7}
 \end{align}
 where (\ref{eq:BiasGradient-7}) uses    $ \E_{\direction}\ilbracket{\loss\ilparenthesis{\btheta}\direction}  = \zero $, $ \E_{\direction}\ilbracket{\direction^\transpose\bg\ilparenthesis{\btheta}\direction} = \bg\ilparenthesis{\btheta} $, and $ \E_{\direction}\ilset{ \ilbracket{\direction^\transpose\bH\ilparenthesis{\btheta}\direction} \direction} = \zero $. 
	\end{itemize}

		The dominant contributor to the asymptotic variance of each element in $\hbg_k$ is the $ O(\perturb_k ^{-1})  $ times the noise  term, leading to a variance that is asymptotically proportional to $ \perturb_k^{-2} $ with constant of proportionality independent of $k$ because $ \Var\ilparenthesis{\lossNoise_k^{+} - \lossNoise_k^{-}  }  $ is asymptotically constant in $k$ assumed in  A.\ref{assume:Noise}.  	Furthermore, each elements at an arbitrary position in the $ O(\perturb_k^{-1}) $ vector, as derived from (\ref{eq:ghat}), are uncorrelated across $k$ by the independence assumptions on $ \ilset{\direction_k} $.  Consequently, $ \E_k\ilparenthesis{\norm{\hbg_k}^2}\stackrel{\mathrm{a.s.}}{=}  O(\perturb_k^{-2})$.  	
		\end{proof}

	\begin{proof} 
		[Proof of Lemma~\ref{lem:BiasHessian}]   
When A.\ref{assume:Noise} and A.\ref{assume:Perturbation} hold,   we have 
$\E_k\ilparenthesis{\hbH_k }\stackrel{\mathrm{a.s.}}{=} \nabla^2\loss_{\perturb_k}\ilparenthesis{\hbtheta_k}$ for estimator (\ref{eq:Hhat-3})  following the proof of Lemma~\ref{lem:BiasGradient}.   When A.\ref{assume:loss3Lips} holds, i.e.,   $ \norm{\nabla^3\loss\ilparenthesis{\btheta} - \nabla^3\loss\ilparenthesis{\bzeta}} \le \LipsPara_3\norm{\btheta-\bzeta} $,   by  mean-value theorem we have  $ |  \loss\ilparenthesis{\bzeta} - \loss\ilparenthesis{\btheta} - \ilparenthesis{\bzeta-\btheta}^\transpose\bg\ilparenthesis{\btheta}  -2^{-1}\ilparenthesis{\bzeta-\btheta}^\transpose\bH\ilparenthesis{\btheta}\ilparenthesis{\bzeta-\btheta} - 6^{-1}\nabla^3\loss\ilparenthesis{\btheta} \bracket{ \ilparenthesis{\bzeta-\btheta}\otimes   \ilparenthesis{\bzeta-\btheta} \otimes  \ilparenthesis{\bzeta-\btheta} } |  \le 
24^{-1} \LipsPara_3 \norm{\bzeta-\btheta}^4 $.
	   Moreover,
		\begin{align}
		& 2\perturb^2 \ilbracket{\nabla ^2 \loss_{\perturb}\ilparenthesis{\btheta}  - \bH\ilparenthesis{\btheta}} \nonumber\\
		&= \E _{\direction} \set{\bracket{  {\loss\ilparenthesis{\btheta+\perturb\direction} +\loss\ilparenthesis{\btheta-\perturb\direction} - 2\loss\ilparenthesis{\btheta}}  } \parenthesis{\direction\direction^\transpose-\bI} } - 2\perturb^2 \bH\ilparenthesis{\btheta} \nonumber\\
		&=    \E     \set{    \bracket{\loss\ilparenthesis{\btheta+\perturb\direction} -  \loss\ilparenthesis{\btheta} - \perturb\direction^\transpose\bg\ilparenthesis{\btheta } - 2^{-1 }\perturb^2 \direction^\transpose\bH\ilparenthesis{\btheta}\direction  - 6^{-1} \perturb^3 \nabla^3\loss\ilparenthesis{\btheta}  \parenthesis{\direction\otimes\direction\otimes\direction}      }     \parenthesis{\direction\direction^\transpose-\bI }   } \nonumber\\
		&\,  +    \E   \set{    \bracket{\loss\ilparenthesis{\btheta-\perturb\direction} -  \loss\ilparenthesis{\btheta} + \perturb\direction^\transpose\bg\ilparenthesis{\btheta } - 2^{-1 }\perturb^2 \direction^\transpose\bH\ilparenthesis{\btheta}\direction   +  6^{-1} \perturb^3 \nabla^3\loss\ilparenthesis{\btheta}  \parenthesis{\direction\otimes\direction\otimes\direction}   }     \parenthesis{\direction\direction^\transpose-\bI }   }\,,       \label{eq:H2}
		\end{align}where (\ref{eq:H2}) uses $ \E_{\direction}\bracket{\direction^\transpose\bg\ilparenthesis{\btheta}\times\ilparenthesis{\direction\direction^\transpose-\bI}} = \zero $,   $  \E_{\direction} \bracket{2^{-1}{\direction^\transpose \bH\ilparenthesis{\btheta}\direction} \times  \parenthesis{\direction\direction^\transpose-\bI}  } =\bH\ilparenthesis{\btheta} $, and $ \E_{\direction} \bracket{ \nabla^3\loss\ilparenthesis{\btheta} \parenthesis{\direction\otimes\direction\otimes\direction}   \times  \parenthesis{\direction\direction^\transpose-\bI } }  = \zero $ when $\direction$ follows A.\ref{assume:Perturbation}. 
Combined with triangle inequality and the higher-moments for multivariate chi-squared distribution, we have 
		\begin{equation*}
		\norm{\nabla ^2 \loss_{\perturb}\ilparenthesis{\btheta}  - \bH\ilparenthesis{\btheta}}\le  \frac{\perturb^2\LipsPara_3}{24}  \E _{\direction} \parenthesis{\norm{\direction} ^4 \norm{\direction\direction^\transpose-\bI }}  = \frac{\LipsPara_3}{24} \perturb^2 \times  O(\Dimension^{\nicefrac{7}{2}})\,. 
		\end{equation*}  
		Therefore, $\E_k\ilparenthesis{\hbH_k}\stackrel{\mathrm{a.s.}}{=}\bH\ilparenthesis{\hbtheta_k} + O(\perturb_k^2)$ for (\ref{eq:Hhat-3}). We can similarly derive  $\E_k\ilparenthesis{\hbH_k}\stackrel{\mathrm{a.s.}}{=}\bH\ilparenthesis{\hbtheta_k} + O(\perturb_k^2)$ for the other three estimators (\ref{eq:Hhat-0}--\ref{eq:Hhat-2}).

		 	The dominant contributor to the asymptotic variance of each element in $\hbH_k$ is the $ O(\perturb_k ^{-2})  $ times the noise  term, leading to a variance that is asymptotically proportional to $ \perturb_k^{-4} $ with constant of proportionality independent of $k$ because $ \Var\ilparenthesis{\lossNoise_k^{+} + \lossNoise_k^{-} - 2 \lossNoise_k }  $ is asymptotically constant in $k$ and because of A.\ref{assume:Noise}.  	Furthermore, each elements at an arbitrary position in the $ O(\perturb_k^{-2}) $ matrix, as derived from (\ref{eq:Hhat}), are uncorrelated across $k$ by the independence assumptions on $ \ilset{\direction_k} $.  Consequently, $ \E_k\ilparenthesis{\norm{\hbH_k}^2}\stackrel{\mathrm{a.s.}}{=}  O(\perturb_k^{-4})$. 
	\end{proof}

	\remove{ 
	
	\begin{proof}
		[Proof of Lemma~\ref{lem:BiasHessianPrime}]
		When A.\ref{assume:loss4Bound} holds, we have $ \loss\ilparenthesis{\bzeta}  =  \loss\ilparenthesis{\btheta} + \ilparenthesis{\bzeta-\btheta}^\transpose\bg\ilparenthesis{\btheta} + 2^{-1} \ilparenthesis{\bzeta-\btheta}^\transpose\bH\ilparenthesis{\btheta} \ilparenthesis{\bzeta-\btheta} + 6^{-1} \nabla^3\loss\ilparenthesis{\btheta} \bracket{\ilparenthesis{\bzeta-\btheta}\otimes\ilparenthesis{\bzeta-\btheta}\otimes \ilparenthesis{\bzeta-\btheta}}  +  O (\norm{\bzeta-\btheta}^4)  $ by Taylor expansion plus the mean-value forms of the remainder. The rest follows similarly as 
the proof for Lemma~\ref{lem:BiasHessian}.
		\end{proof}
	
}

	\begin{proof}
		[Proof of Thm.~\ref{thm:ParameterConvergence}]
	From Lemma~\ref{lem:BiasGradient}, we have $ \E_k\ilparenthesis{\hbg_k}  = \bg\ilparenthesis{\hbtheta_k} + \bias_k
		$ with $ \perturb_k^{-2} \norm{\bias_k} $ being uniformly bounded for sufficiently large $k$. The rest follows from \citep[Thm. 1a]{spall2000adaptive}.  
		Note that the proof for \citep[Thm. 1a]{spall2000adaptive} does not assume any particular form of the Hessian estimate; it only requires  A.\ref{assume:PDmapping} and the result in Lemmas~\ref{lem:BiasGradient}--\ref{lem:BiasHessian}. 	
	\end{proof}

		\begin{proof}[Proof of Thm.~\ref{thm:HessianConvergence}]
		The proof proceeds similarly as \citep[Appendix A]{zhu2020error}.  Let $\bW_k \equiv \hbH_k - \E_k\ilparenthesis{\hbH_k} $, which satisfies  $\E\bW_k=\zero$ for all $k$. 
		Thanks to      Lemma~\ref{lem:BiasHessian}, $ \E\ilparenthesis{\perturb_k^4\norm{\hbH_k}^2} <\infty $ uniformly for  all $k$. 
		When  A.\ref{assume:Stepsize} holds,  we have $ \sum_k \parenthesis{   {\E\norm{\gainH_k \bW_k}^2}   }<\infty $. By  martingale convergence theorem,  we have $ \sum_k \gainH_k \bW_k \to \zero $ a.s.  Now  that 
		Lemma~\ref{lem:BiasHessian} gives $ \E_k\ilparenthesis{\hbH_k }  = \bH\ilparenthesis{\hbtheta_k} + O(\perturb_k^2) $, we have $ \sum_k\ilbracket{ \gainH_k \E_k \ilparenthesis{\hbH_k} } = \sum_k \ilset{ \gainH_k \ilbracket{\bH\ilparenthesis{\hbtheta_k} + O(\perturb_k^2)} }  \to \bH\ilparenthesis{\bvartheta} $ a.s., using  A.\ref{assume:loss3Lips} (Hessian is continuous around $\hbtheta_k$),  Theorem~\ref{thm:ParameterConvergence} ($\hbtheta_k $ converges a.s. to $\bvartheta$), and A.\ref{assume:Stepsize} ($ \gainH_k \perturb_k^2 \to 0  $). 
	\end{proof}

\begin{proof}[Proof of Thm.~\ref{thm:Normality}]
	The proof can be derived following the proofs of  the general result \citet[Thm.3]{zhu2020harp} after specifying the conditioner and the distribution for the perturbation.   We skip the sketch here as it involves quite a few    notations not introduced here.  
	\remove{
Before stating Theorem~\ref{thm:Normality}, we state the consequences of previous Lemmas and Theorems. Write the conditioner as $ \bPsi_k = \ilbracket{ \mappingPD_k\ilparenthesis{\obH_k}}^{-1} $, and $ \bPsi = \lim_{k\to\infty}\bPsi_k = \ilbracket{\bH\ilparenthesis{\bvartheta}}^{-1} $ per Theorem~\ref{thm:HessianConvergence}. 
Let $\bGamma_k \equiv \gain \bPsi_k \bH\ilparenthesis{\obtheta_k}$ with $\obtheta_k$ being some convex combination of $\hbtheta_k$ and   $\bvartheta$, and $ \bGamma = \lim_{k\to\infty} \bGamma_k  =  \gain\bI $ per Theorems~\ref{thm:ParameterConvergence}--\ref{thm:HessianConvergence}.  
We also 
introduce some notation. 
We enforce the form of the stepsizes: $ \gain_k = \nicefrac{\gain}{(k+1)^{\upalpha}} $, $ \perturb_k=\nicefrac{\perturb}{(k+1)^{\upgamma}} $, and $ \gainH_k $ satisfies A.\ref{assume:Stepsize}.  Let $\uptau\equiv \upalpha-2\upgamma$ and $ \uptau_+ \equiv \uptau\cdot\indicator_{\ilset{\upalpha=1}} $. 
Let the bias term of the gradient estimator be $ \bias_k\equiv  \E\ilbracket{  \hbg_k - \bg\ilparenthesis{\hbtheta_k} }  $, and let the corresponding noise term be $ \noise_k\equiv  \hbg_k - \E_k\ilparenthesis{\hbg_k} $.  
Let $\bm{t}_k = -\gain\ilparenthesis{k+1}^{\nicefrac{\uptau}{2}}  \bPsi_k  \bias_k$ and $\bv_k = -\gain \ilparenthesis{k+1}^{-\upgamma} \noise_k
$. 	
}
	\end{proof}

\end{document}